\documentclass[sn-mathphys]{Paper}
\usepackage{amscd} 
\usepackage{verbatim}
\usepackage{amsxtra}
\usepackage{amsthm}
\usepackage{amsmath,etoolbox}
\usepackage{amssymb}
\usepackage{paralist}
\usepackage{graphics} 
\usepackage{epsfig} 
\usepackage{subcaption}
\usepackage{graphicx}  
\usepackage{epstopdf}
\usepackage{mathtools}
\usepackage{centernot}
\usepackage[shortlabels]{enumitem}
\usepackage{extarrows}
\usepackage{mathrsfs}
\hypersetup{urlcolor=blue, citecolor=red}
\usepackage{microtype}
\usepackage[mathscr]{euscript}
\usepackage{float}
\usepackage{csquotes}

\makeatletter
\newenvironment{sqcases}{%
  \matrix@check\sqcases\env@sqcases
}{%
  \endarray\right.%
}
\def\env@sqcases{%
  \let\@ifnextchar\new@ifnextchar
  \left[
  \def\arraystretch{1.2}%
  \array{@{}l@{\quad}l@{}}%
}
\makeatother

\makeatletter
\newcommand{\raisemath}[1]{\mathpalette{\raisem@th{#1}}}
\newcommand{\raisem@th}[3]{\raisebox{#1}{$#2#3$}}
\makeatother

\def\f{\longrightarrow}

\def\N{\mathbb{N}}

\def\e{\varepsilon}
\def\l{\lambda}

\def\x{\bar{x}}

\def\y{\bar{y}}
\def\u{\bar{u}}
\def\p{\bar{p}}
\def\xxi{\bar{\xi}}
\def\nnu{\bar{\nu}}
\def\zzeta{\bar{\zeta}}
\def\ttheta{\bar{\theta}}
\def\vvartheta{\bar{\vartheta}}
\def\oomega{\bar{\omega}}
\def\a{\alpha}
\def\b{\beta}
\def\<{\langle}
\def\>{\rangle}

\def\e{\varepsilon}
\def\R{\mathbb{R}}
\def\O{\mathcal{O}}

\def\inte{\textnormal{int}\,}
\def\clo{\textnormal{cl}\,}
\def\epi{\textnormal{epi}\,}
\def\bdry{\textnormal{bdry}\,}
\def\dom{\textnormal{dom}\,}
\def\conv{\textnormal{conv}\,}

\def\Gr{\textnormal{Gr}\,}
\def\gk{\gamma_k}
\def\gkn{\gamma_{k_n}}
\def\Vphi{\varPhi}

\def\CO{\mathcal{C}}

\def\-{\textnormal{-}}
\newcommand*{\tran}{^{\mkern-1.5mu\mathsf{T}}}

\DeclareMathOperator{\sign}{sign}
\def\sp{\hspace{0.015cm}}
\def\bp{\hspace{-0.08cm}}
\def\bbp{\hspace{-0.04cm}}

\def\bbbp{\hspace{-0.02cm}}
\def\TV{_{\textnormal{\tiny{T.\sp V.}}}}

\def\0gk{\scaleto{0,\gk}{3.5pt}}
\def\supp{\textnormal{supp}\,}
\makeatletter

\usepackage{scalerel}
\theoremstyle{thmstyleone}%
\newtheorem{theorem}{Theorem}[section]% meant for sectionwise numbers
\newtheorem{proposition}[theorem]{Proposition}% 
\newtheorem{lemma}[theorem]{Lemma}
\newtheorem{corollary}[theorem]{Corollary}

\theoremstyle{thmstyletwo}%
\newtheorem{remark}[theorem]{Remark}%

\theoremstyle{thmstylethree}%

\begin{document}

\title[Control Space for Strong Convergence of Continuous Approximation]{A Control Space Ensuring the Strong Convergence of Continuous Approximation for a Controlled Sweeping Process}

\author[1]{\fnm{Chadi} \sur{Nour}}\email{cnour@lau.edu.lb}
\equalcont{These authors contributed equally to this work.}

\author*[2]{\fnm{Vera} \sur{Zeidan}}\email{zeidan@msu.edu}
\equalcont{These authors contributed equally to this work.}

\affil[1]{\orgdiv{Department of Computer Science and Mathematics}, \orgname{Lebanese American University}, \orgaddress{\street{Byblos Campus}, \postcode{P.O. Box 36}, \state{Byblos}, \country{Lebanon}}}

\affil[2]{\orgdiv{Department of Mathematics}, \orgname{Michigan State University}, \orgaddress{\city{East Lansing}, \postcode{48824-1027}, \state{MI}, \country{USA}}}


\abstract{A controlled sweeping process with prox-regular set, $W^{1,2}$-controls, and separable endpoints constraints is considered in this paper. Existence of optimal solutions is established and {\it local} optimality conditions are derived via {\it strong} converging {\it continuous} approximations  that  entirely reside in the {\it interior} of the prox-regular set. Consequently, these results are expressed in terms of new subdifferentials for the original data that are {\it strictly smaller} than the standard Clarke and Mordukhovich subdifferentials.}




\maketitle

\section{Introduction}\label{intro} This paper addresses the following fixed time Mayer-type optimal control problem involving $W^{1,2}$-controlled sweeping systems
$$\begin{array}{l} (P)\colon\; \hbox{Minimize}\;g(x(0),x(1))\vspace{0.1cm}\\ \hspace{0.9cm} \hbox{over}\;(x,u)\in AC([0,1];\R^n)\times \mathscr{W} \;\hbox{such that}\;\\[2pt] \hspace{0.9cm}  \begin{cases} (D)  \begin{sqcases}\dot{x}(t)\in f(x(t),u(t))-\partial \varphi(x(t)),\;\;\hbox{a.e.}\;t\in[0,1],\\x(0)\in C_0\subset \dom \varphi, \end{sqcases}\vspace{0.1cm}\\ x(1)\in C_1, \end{cases}
    \end{array}$$ where, $g\colon\R^n\times\R^n\f\R\cup\{\infty\}$, $f\colon\R^n\times \R^m \f\R^n$, $\varphi\colon\R^n\f\R\cup\{\infty\}$, $\partial$ stands for the Clarke subdifferential, $C:=\dom\varphi$ is the zero-sublevel set of a function $\psi\colon\R^n\f\R$, that is, $C=\{x\in\R^n : \psi(x)\leq0\},$ $C_0\subset C$, $C_1\subset\R^n$, and, for $U\colon[0,1]\rightrightarrows\R^{m}$ a multifunction and $ \mathbb{U}:=\bigcup_{t\in[0,1]} U(t)$,  the set of control functions $\mathscr{W}$ is defined by
\begin{equation}\label{setw}\hspace{-0.2cm}
\mathscr{W}:={W}^{1,2}([0,1]; \mathbb{U})=\left\{ u\in {W}^{1,2}([0,1]; \R^m): u(t) \in U(t),\;\; \forall\sp t\in [0,1]\right\}. \vspace{-0.15cm}
\end{equation}
Note that if $(x,u)$ solves  ($D$), it necessarily follows that $x(t)\in C$ for all $t\in[0,1]$.

A pair $(x,u)$ is {\it admissible} for $(P)$ if $x\colon[0,1]\f\R^n$ is absolutely continuous, $u\in \mathscr{W}$, and $(x,u)$  satisfies the {\it perturbed} and {\it controlled\,} {\it sweeping process} $(D)$, called the {\it dynamic} of $(P)$. 

An admissible pair $(\x,\u)$ for $(P)$ is said to be a $W^{1,2}$-{\it local minimizer} (also known as {\it intermediate local minimizer of rank} 2) if there exists $\delta >0$ such that  \begin{equation}\label{optimal}g(\x(0),\x(1))\le g(x(0),x(1)), \end{equation}
for all $(x,u)$ admissible for $(P)$ with $\|x- \x\|_{\infty}\leq\delta$, $\|\dot{x}- \dot{\x}\|_{2}^2\leq\delta$, $\|u- \u\|_{\infty}\leq\delta$ and  $\|\dot{u}- \dot{\u}\|_{2}^2\leq\delta.$  Note that if \eqref{optimal} is satisfied for any admissible pairs $(x,u)$, then $(\x,\u)$ is called a {\it global minimizer} (or an {\it optimal solution}\sp) for $(P)$.

{\it Sweeping processes} first appeared in the papers \cite{moreau1,moreau2,moreau3} by J.J. Moreau in the context of friction and plasticity theory. Since then, these systems have emerged in further applications such as hysteresis, ferromagnetism, electric circuits, phase transitions, economics, etc. (see e.g., \cite{outrata} and its references). The main feature of such systems is the presence in its dynamic of a normal cone to a set $C$ that is, the subdifferential of the {\it indicator function} of $C$,  or more generally, the subdifferential of an extended-real-valued function. Consequently, the dynamic is  a differential inclusion with unbounded and discontinuous right-hand side. Therefore, sweeping processes fall outside the scope of {\it standard} differential inclusions, and hence, studying optimal control problems over this model requires creative new techniques
 
 In \cite{brokate,pinho,pinholast,verachadiimp,verachadi} (see also \cite{pinhoEr}), necessary optimality conditions in the form of a {\it maximum principle} for optimal control problems involving {\it measurably}-controlled sweeping processes are derived using a {\it smooth penalty-type} approximations, called here {\it continuous} approximations. A special feature of \cite{pinho,pinholast,verachadiimp,verachadi}, which led in \cite{pinhonum,verachadinum} to numerical algorithms, is the {\it novel} exponential penalization technique that approximates the dynamic of $(D)$ by the following sequence of {\it standard}  control systems $$(D_{\gk})\;\;\;\dot{x}(t)=f(x(t),u(t))-\nabla\Vphi(x(t))-\gk e^{\gk\psi(x(t))} \nabla\psi(x(t)),\;\textnormal{a.e.}\; t\, \in[0,1], $$ where $\Vphi$ is a smooth extension to $\R^n$ of $\varphi$, and $\gk$ is a positive sequence that converges to $\infty$ as $k\f\infty$. 
 In these papers,  necessary optimality conditions are developed  via    approximating {\it weakly} the optimal solution of $(P)$  by a sequence of optimal solutions for standard optimal control problems over $(D_{\gk})$.  That is, the velocity sequence  of the optimal state  for the approximating problems convergences {\it weakly} in $L^2$ to the given velocity of the solution for $(P)$.

  {\it Strong} convergence of velocities is well-known to be an essential property  for numerical purposes,  as pointed out in several papers, see e.g., \cite{cmo,cmo2,chhm,pinho}.  In other words, it is important  that the solutions of    $(P)$ be  {\it strongly} approximated (in the 
$W^{1,2}$-norm) by the solutions of  approximating problems  that  are computable via existing numerical algorithms.  This question of strong convergence of velocities was previously addressed using {\it discrete} approximations, see for instance, \cite{ccmn,cmo0,cmo, cmo2, chhm2, chhm, cmn0},  where the authors considered  optimal control problems involving {\it various} forms of controlled sweeping processes including the $W^{1,2}$-controls.  In \cite{cmo0,cmo,cmo2},  this approach also served   to derive necessary optimality conditions phrased in terms of the weak-Pontryagin-type maximum principle when the control space is $W^{1,2}([0,1];\R^m).$  Therein, these optimality criteria are applied to  {\it real-life} models, whose optimal controls turn out to be $W^{1,2}$.
  
 The main goal of the paper is motivated by the importance of approximating a solution of the sweeping process ($D$) by  solutions of $(D_{\gk})$ whose velocity  {\it strongly converges} to the  velocity of the solution of ($D$) as described above. We establish the validity of this result  when the controls in  ($D_{\gk}$) are chosen to be $W^{1,2}$. As a consequence, we embark on the study of the problem $(P)$. We first show that, under suitable conditions, the problem $(P)$ admits an optimal solution $(\x,\u)$. Then, we approximate a given optimal solution ($\x,\u$) by a sequence of optimal solutions for standard optimal control problems over $(D_{\gk})$ with objective functions  carefully formulated to guarantee the {\it strong} convergence of the solution velocities. To our knowledge, this is a {\it first-of-its-kind} result that uses {\it continuous} approximations, as opposed to discrete approximations, to obtain strong convergence of velocities. Furthermore, necessary optimality conditions are established for $W^{1,2}$-local minimizers of $(P)$ upon taking the limit of the optimality conditions for the corresponding approximating optimal control problems. This latter task requires meticulous analysis.
 
The paper is organized as follows. Notations and some definitions from nonsmooth analysis will be given in the next section. In Section \ref{hypo}, we provide a list of assumptions and some important consequences. Moreover, we present some needed results from \cite[Sections 4\,$\&$\,5]{verachadiimp} including the connection between $(D_{\gk})$ and $(D)$ under measurable controls. Section \ref{mainresults} is devoted to $(i)$ showing that $(D_{\gk})$ strongly approximates $(D)$ when $W^{1,2}$-controls are utilized, $(ii)$  establishing an existence theorem for an  optimal solution of $(P)$, $(iii)$  constructing for $(P)$ a continuous approximating sequence of standard optimal control problems $(P_{\gk})$, 
and $(iv)$ deriving necessary optimality conditions in the form of weak-Pontryagin-type  maximum principle for $W^{1,2}$-local minimizers of $(P)$. 
To maintain an easy  flow of the main results,  most of the proofs  are provided in  Section \ref{auxresults}, where we also establish  some auxiliary results  employed in different places of the paper.
%-----------------------------------------------------------------------------------------------------------------

\section{Preliminaries}\label{prelim} 
\subsection{Basic notations}
In the sequel, the notations used in this paper are provided. By $\|\cdot\|$, $\<\cdot,\cdot\>$, $B$ and $\bar{B}$, we denote, respectively, the Euclidean norm, the usual inner product, the open unit ball and the closed unit ball. An open (resp. closed) ball of radius $\rho>0$ and centered at $x\in\R^n$ is written as $B_{\rho}(x)$ (resp. $\bar{B}_{\rho}(x)$). For $x,\,y\in\R^n$, $[x,y]$ and $(x,y)$ denote, respectively, the closed and the open line segment joining $x$ to $y$. For a set $C \subset \R^n$, $\inte C$, $\bdry C$, $\clo C$, $\conv C$, $C^c$, and $C^\circ$ designate  the interior, the boundary, the closure, the convex hull, the complement,  and the polar of $C$, respectively. The distance from a point $x$ to a set $C$ is denoted by $d(x,C)$. For an extended-real-valued function $\varphi\colon\R^n\f\R\cup\{\infty\}$,  $\dom \varphi$ is the effective domain of $\varphi$ and $\epi \varphi$ is its epigraph. For  a multifunction $F\colon\R^n\rightrightarrows\R^{m}$, $\Gr F\subset \R^n\times\R^m$ denotes the graph of $F$, that is, $\Gr F:=\{(x,v)\in\R^n\times\R^m : v\in F(x)\}$. The Lebesgue space of $p$-integrable functions $h\colon [a,b]\f\R^n$ is denoted by $L^p([a,b];\R^n)$ or simply $L^p$ when the domain and range are clearly understood. The norms in $L^p([a,b];\R^n)$ and $L^{\infty}([a,b];\R^n)$ (or $C([a,b];\R^n)$) are written as $\|\cdot\|_p$ and $\|\cdot\|_{\infty}$, respectively. The set of all absolutely continuous functions from an interval $[a,b]$ to $\R^n$ will be denoted by $AC([a,b];\R^n)$, and  $\mathscr{M}_{m\times n}([a, b])$ is  the set of $m\times n$-matrix functions  on $[a,b]$. We say that $h$ is a $BV$-function, and we write $h\in BV([a,b];\R^n),$ if  $h$ is a function of bounded variation, that is, $V_a^b(h)<\infty$, where  $V_a^b(h)$ is the total variation of $h$. We denote by $NBV[a,b]$ the normalized space of $BV$-functions on $[a,b]$ that consists of those $BV$-functions $\Omega$ such that $\Omega(a)=0$ and $\Omega$ is right continuous on $(a,b)$ (see e.g., \cite[p.115]{luenberger}). The set $C^*([a,b]; \R)$ denotes the dual of $C([a,b];\R),$ equipped with the supremum norm. The norm on $C^*([a,b]; \R)$, denoted by $\|\cdot\|\TV$, is the induced norm. By Riesz representation theorem, the elements of $C^*([a,b]; \R)$ are interpreted as belonging to $\mathfrak{M}([a,b])$, the set of finite signed Radon measures on $[a,b]$ equipped with the weak* topology. Thereby,  to each element of $C^*([a,b]; \R)$ it corresponds a unique element in $NBV[a,b]$ related through the Stieltjes integral and both elements  have the same total variation. We denote by $C^{\oplus}(a,b)$  the subset of  $C^*([a,b];\R)$  taking nonnegative values on nonnegative-valued functions in $C([a,b]; \R)$. For a compact subset $S\subset\R^d$, $C(S;\R^n)$ designates the set of  continuous functions from  $S$ to $\R^n$. We denote by $W^{k,p}([a,b];\R^n)$, $k\in\N$ and $p\in [0,+\infty]$, the classical Sobolev space. Note that in this paper, the Sobolev space $W^{1,2}([a,b];\R^n)$ will be considered with the norm $\|x(\cdot)\|_{W^{1,2}}:=\|x(\cdot)\|_\infty + \|\dot{x}(\cdot)\|_2.$ Hence, the convergence of a sequence $x_n$ strongly in the norm topology of the space $W^{1,2}([a,b];\R^n)$ is equivalent to the uniform convergence of $x_n$ on $[a,b]$ and the strong  convergence  in $L^2$ of its derivative $\dot{x}_n$.  Finally, a function $F\colon\R^n\to \R$ is $\mathcal{C}^{1,1}$ if it is Fr\'echet differentiable  with locally Lipschitz derivative.

\subsection{ Notions in nonsmooth analysis}

We begin by listing  standard notions and facts for which the reader is invited to
 consult the monographs \cite{clsw}, \cite{mordubook}, and \cite{rockwet}. Let $C$ be a nonempty and closed subset of $\R^n.$ For $x\in C$, the {\it proximal}, the {\it Mordukhovich} ({also known as \it limiting}) and the {\it Clarke normal} cones to $C$ at $x$ are denoted by $N^P_C(x)$, $N_C^L(x)$ and $N_C(x)$, respectively. Using \cite[Proposition 1.1.5(b)]{clsw}, we deduce that these three normal cones enjoy an essential {\it local property}, namely, if two closed sets in $\R^n$ are the {\it same} in a {\it neighborhood} of $x$, then these two sets possess at $x$ the {\it same normal cone} (proximal, Mordukhovich, or Clarke). An important feature for the Mordukhovich normal cone is that the multifunction $N^L_C(\cdot)$ has closed values and a closed graph. On the other hand, when $C$ is convex then the proximal, the Mordukhovich and the Clarke normal cones to $C$ coincide with the well-known normal cone to convex sets. 

 For $\rho > 0$, the set $C$ is said to be $\rho$-\textit{prox-regular} whenever the {\it proximal normal inequality}, see \cite[Proposition 1.1.5(a)]{clsw}, holds for $\sigma=\frac{1}{2\rho}$, for all $x\in C$ and for all $\zeta$ unit in $N_C^P(x)$. In particular, every convex set is $\rho$-prox-regular for every $\rho > 0$, and every compact set with a $\CO^{1,1}$-boundary is $\rho$-prox-regular, where $\rho$ depends on the Lipschitz constant of the gradient of the boundary parametrization. Note that, for a $\rho$-prox-regular set $C,$ we have  $N_C^P (x) = N_C^L (x)  = N_C (x)$ for all $x \in \R^n.$  For more information about prox-regularity, and related properties such as {\it positive reach}, {\it proximal smoothness}, {\it exterior sphere condition} and $\varphi_0$-{\it convexity}, see \cite{prox,cm,fed,nst,prt}. 

The following  geometric properties shall be used in the rest of the paper. A closed set $A\subset\R^n$ is said to be {\it epi-Lipschitz} (or {\it wedged}\,) at a point $x\in A$ if the set $A$ can be viewed near $x$, after application of an orthogonal matrix, as the epigraph of a Lipschitz continuous function. If this holds for all $x\in A$, then we simply say that $A$ is epi-Lipschitz. This geometric definition was introduced by Rockafellar in \cite{rock79}. The epi-Lipschitz property of $A$ at $x$ is also characterizable in terms of the nonemptiness of the topological interior of the Clarke tangent cone of $A$ at $x$ which is also equivalent to the {\it pointedness} of the Clarke normal cone of $A$ at $x$, that is, $N_A(x)\cap-N_A(x)=\{0\}$, see \cite[Theorem 7.3.1]{clarkeold}. Note that a {\it convex} set is epi-Lipschitz if and only if it has a nonempty interior. For more information about this property, see \cite{clarkeold,clsw,rockwet}. On the other hand, a set $A\subset\R^n$ is  {\it quasiconvex} if there exists $\a\geq 0$ such that any two points $x,\,y$ in $A$ can be joined by a polygonal line $\gamma$  in $A$ satisfying $l(\gamma) \leq \alpha \|x-y\|,$ where $l(\gamma)$ denotes the length of $\gamma$. In this paper, the quasiconvexity of $C$ is vital  for extending our function $\varphi$ from $C$ to $\mathbb{R}^n$ while preserving special properties (see Lemma \ref{prop1}). For more explanation about this notion, consult \cite{brudnyi}. 

Next, we recall  the {\it standard} notions of  proximal, Mordukhovich, and Clarke subdifferentials, and Clarke generalized Jacobian and Hessian  (see \cite{clarkeold,clsw, mordubook, rockwet}).  We also enlist the {\it nonstandard} notions for subdifferentials  introduced and studied in \cite{verachadiimp,verachadi} that  are instrumental for this paper  for being  strictly smaller than their standard counterparts notions.
 
 For the standard notions, given a {\it lower semicontinuous} function $G\colon\R^n\f\R\cup\{\infty\}$, and  $x \in \dom G$, the \textit{proximal}, the {\it Mordukhovich} (or \textit{limiting}) and the \textit{Clarke subdifferential}  of $G$ at  $x$ are denoted by  $\partial^P G (x)$, $\partial^L G(x)$ and $\partial G (x)$, respectively. From the properties of the limiting normal cone,  $\partial^LG(\cdot)$ has a closed graph and closed values. Note that if  $x\in\inte (\dom{G})$ and $G$ is Lipschitz near $x$,  \cite[Theorem 2.5.1]{clarkeold} yields that the  Clarke subdifferential of $G$ at  $x$ coincides  with the {\it Clarke generalized gradient} of ${G}$ at $x$, also denoted here by $\partial {G} (x)$. If ${G}$ is $\CO^{1,1}$ near $x\in \inte(\dom G)$,  $\partial^2G(x)$ denotes the {\it Clarke generalized Hessian} of ${G}$ at $x$. For  $H\colon\R^n \f \R^n$ Lipschitz near $x\in\R^n$, the {\it Clarke generalized Jacobian} of $H$ at $x$ is denoted by $\partial{H}(x)$.
 
For the nonstandard notions, given a {\it lower semicontinuous} function $G\colon\R^n\f\R\cup\{\infty\}$,  $S\subset \dom G$ closed with $\inte S\neq\emptyset$, and $x\in\clo(\inte S)$,  we define  the \enquote{{\it limiting subdifferential} of $G$ {\it relative} to $\inte S$ at the point $x$}  to be 
  \begin{equation} \label{delelel} \partial^L_\ell G(x):=\left\{\lim_i \zeta_i : \zeta_i\in\partial^P  {G}(x_i),\; x_i\in \inte S,\;\hbox{and}\;x_i\xrightarrow{{G}} x\right\}.\end{equation}
  where $x_i\xrightarrow{{G}} x$ signifies that $x_i\f x$ and ${G}(x_i)\f G(x)$. If  $\dom G$ is {\it closed}, $\inte(\dom G)\neq \emptyset$,  and $G$  is {\it locally Lipschitz} on $\inte (\dom{G})$, then for $x\in\clo(\inte (\dom{G}))$, we define the \enquote{{\it extended Clarke generalized gradient} of $G$ at $x$}, denoted by $\partial_\ell G(x)$, to be  
\begin{equation}\label{phil1}\partial_\ell G(x):=\conv \left\{\lim_{i\f\infty} \nabla {G}(x_i) : x_i\xrightarrow{\O} x \; \hbox{and}\; \nabla G(x_i) \;\hbox{exists}\; \forall i\;\right\},
\end{equation} 
where $\O$ is any full-measure subset of $\inte (\dom{G})$. If ${G}$ is $\CO^{1,1}$ on $\inte(\dom G)$ and  $x\in\clo(\inte (\dom{G}))$, we define similarly to $\partial_\ell G(x)$ the \enquote{{\it extended  Clarke generalized Hessian} of $G$ at $x$} to be
\begin{equation}\hspace{-0.2cm}\label{phil2} \partial^2_\ell{G}(x):=\conv \left\{\lim_{i\f\infty} \nabla^2 {G}(x_i) : x_i\xrightarrow{\O} x\; \mbox{and}\; \nabla^2 G(x_i) \;\hbox{exists}\; \forall i\;\right\},\end{equation}
 where $\O$ is any full-measure subset of $\inte (\dom{G})$. For $S\subset \R^n$ closed with $\inte S\not=\emptyset$, if $G$ is $\CO^{1,1}$ on an open set containing  $S$,  then  for $x\in S$  we define the \enquote{{\it Clarke generalized Hessian} of $G$ {\it relative} to  $\inte S$ at $x$} to be
\begin{equation}\label{phil4}
\partial^2_\ell{G}(x):= \conv \left\{\lim_{i\f\infty} \nabla^2 {G}(x_i) : x_i\xrightarrow{\O} x\right\}, 
\end{equation}
where $\O$ is any full-measure subset of ${\inte} S$.  Now let $S\subset \R^n$ be closed with $\inte S\not=\emptyset$, and let $H\colon\R^n \f \R^n$ be locally Lipschitz on $\inte S$. Then for $x\in\clo(\inte S)$, we defined the \enquote{{\it extended Clarke generalized Jacobian} of $H$ at $x$} to be
\begin{equation}\label{phil3}
\partial_\ell{H}(x):= \conv \left\{\lim_{i\f\infty} J {H}(x_i) : x_i\xrightarrow{\O} x\;  \hbox{and}\; J{H}(x_i) \;\hbox{exists}\; \forall i\right\}, 
\end{equation}
where $\O$ is any full-measure subset of $\inte S$ and $J$ is the Jacobian operator.  Finally,  when the set defined in (\ref{phil1}), (\ref{phil2}), (\ref{phil3}) or (\ref{phil4}) is a singleton, then we shall use the notations $\nabla_{\bp\ell}$ and $\nabla_{\bp\ell}^2$ instead of $\partial_\ell$ and $\partial_\ell^2$, respectively.

\section{Assumptions, consequences, and known results}\label{hypo} In this section, hypotheses on the data of $(P)$ are introduced and some of their important consequences are provided.  We also present some needed results from \cite[Sections 4\,$\&$\,5]{verachadiimp} where the connection between $(D_{\gk})$ and $(D)$ under measurable controls is studied. We note that each result of this paper shall require a {\it selected}  group of these hypotheses. Furthermore, a local version of (A1) is stated and used in the relevant locations of the paper.
\begin{enumerate}[label=\textbf{A\arabic*}:]
\item There exist $ M>0$  and $\tilde{\rho}>0$ such that $f$ is $M$-Lipschitz on $C\times (\mathbb{U}+\tilde{\rho}\bar{B})$ with 
 $\|f(x,u)\| \leq M $  for all $(x,u)\in C\times (\mathbb{U}+\tilde{\rho}\bar{B})$. 
\item The set $C:=\dom \varphi$ is given by  $C=\{x\in\R^n : \psi(x)\leq 0\},$ where $\psi:\R^n\f\R.$
\begin{enumerate}[label=\textbf{A2.\arabic*}:]
\item There exists $\rho>0$ such that $\psi$ is $\CO^{1,1}$ on $ C+\rho{B}$.
\item There is a constant $\eta>0$ such that $\|\nabla\psi (x)\|>2\eta\;$ for all $x: \psi(x)=0.$
\item The function $\psi$ is coercive, that is, $\lim_{\|x\|\f\infty} \psi(x)=+\infty.$
\item The set $C$ has a connected interior.\footnote{When $\varphi$ has a suitable extension to $\R^n$, as is the case for $\varphi$ being the {\it indicator} of $C$, see Remark \ref{connectednewD}, this condition is omitted. 
%Note that our imposed assumptions on $\psi$ do not ensure the connectedness of $C$, as illustrated by the {\it Cassini oval} case in which
% $$ C:=\{(x,y) : (x^2+y^2)^2-4(x^2-y^2)+1\leq 0\}.$$ 
 }
\end{enumerate}
\item The function $\varphi$ is globally Lipschitz on $C$ and $\CO^{1}$ on $\inte C$. Moreover, the function $\nabla\varphi$ is globally Lipchitz on $\inte C$.
\item For the sets $C_0$, $C_1$, and $U(\cdot)$  we have:\begin{enumerate}[label=\textbf{A4.\arabic*}:]
\item The set $C_0\subset C$ is nonempty and closed.
\item The graph of $U(\cdot)$ is a $\mathscr{L}\times\mathscr{B}$ measurable set,  and,  for $t\in [0,1],$ $U(t)$ is closed, and bounded uniformly in $t$.
\item The set $C_1\subset\R^n$ is nonempty and closed. 
\item The multifunction $U(\cdot)$ is lower semicontinuous.
\end{enumerate}
\end{enumerate}

\begin{remark}\label{epsilon} The coercivity of $\psi$ in (A2.3) is only assumed to obtain the boundedness of the closed set $C$, and hence, this condition can be replaced  by $C$ bounded. On the other hand, for $C\subset \R^n$ defined as the sub-level set of a function $\psi$, one can show that: 
\begin{enumerate}[$(i)$,leftmargin=0.7cm] 
	\item  Whenever $C$ is nonempty and compact, $\psi$ is merely 
	$\mathcal{C}^{1}$  on $C +\rho B$, and (A2.2)  holds, then there exists $\varepsilon >0$ such that	
\begin{equation} \label{velo0} x\in C\;\hbox{and}\; \|\nabla\psi(x)\|\leq \eta\implies \psi(x)<-\e.
\end{equation}
\item When  $\psi$ is merely $\mathcal{C}^1$ on $C+\rho B$  and satisfies \textnormal{(A2.2)-(A2.3)}, by {\cite[Lemma 3.3]{verachadiimp}},
\begin{enumerate}[leftmargin=*]
\item $\bdry C\neq \emptyset$ \;and\; $\bdry C=\{x\in\R^n : \psi(x)=0\},$ 
\item $\inte C\neq \emptyset$ \; and\; $\inte C=\{x\in\R^n : \psi(x)<0\}.$
\end{enumerate}
\end{enumerate}
\end{remark}

 The following important properties of the compact set $C$  were  obtained in  \cite[Proposition 3.1]{verachadi}, where  $\psi$ is assumed to be $\CO^{1,1}$ on {\it all} of $\R^n$.  However,  a slight  modification in the proof of that proposition  is performed  in \cite{verachadiimp} to conclude that these properties are actually valid under our assumption (A2.1).
 
 Here and throughout the paper, $\bar{M}_\psi$ denotes an upper bound of $\|\nabla\psi(\cdot)\|$ on the compact set $C$, and  $2 M_\psi$ is a {\it Lipschitz constant} of $\nabla \psi(\cdot)$ over the compact set $C+\frac{\rho}{2}\bar{B}$ chosen large enough so that $
 M_\psi\geq \frac{4\eta}{\rho}.$ 
\begin{lemma}\textnormal{{\cite[Lemma 3.4]{verachadiimp}}} \label{prop1} Under \textnormal{(A2.1)-(A2.3)}, we have the following\sp$:$
\begin{enumerate}[$(i)$,leftmargin=0.8cm] 
\item The nonempty  set  $C$ is compact, amenable $($in the sense of \textnormal{\cite{rockwet}}$)$, epi-Lipschitzian, $C=\clo(\inte C)$, and  $C$ is $\frac{\eta}{M_\psi} $-prox-regular.
\item For all $x\in\bdry C$ we have $N_C(x)=N^P_C(x)=N^L_C(x)=\{\lambda\nabla\psi(x) : \lambda\geq 0\}.$
\item If also \textnormal{(A2.4)} holds, then $\inte C$ is quasiconvex. Furthermore, if in addition \textnormal{(A3)} is satisfied, then there exists a function $\Vphi\in \CO^{1}(\R^n)$ such that\sp$:$ \begin{itemize}[leftmargin=*]
 \item $\Vphi$ is bounded on $\R^n$, and \;$\Vphi (x)= \varphi (x)$ for all $x\in C$.
 \item $\Vphi$ and $\nabla\Vphi$ are globally Lipschitz on $\R^n$.
 \item  For all $x\in C$ we have  
\begin{equation}\label{eq0}\partial \varphi(x)=\{\nabla \Vphi(x)\} + N_C(x).
\end{equation}	
 \end{itemize}
\end{enumerate}
\end{lemma}

\begin{remark}\label{connectednewD} \qquad  \begin{enumerate}[$(i)$,leftmargin=0.7cm] 
\item In Lemma \ref{prop1}$(iii)$, assumption (A2.4) is  only imposed to ensure the quasiconvexity of $C$ needed to obtain the extension  $\Vphi$ of $\varphi$. Consequently, when such an extension is readily available, condition (A2.4) is discarded. This is the case, for instance, when $\varphi$ is constant on $C$. Thus,  when $\varphi$ is the {\it indicator} function of $C$ assumption (A2.4) is not required. 
\item From Lemma \ref{prop1}$(iii),$  $\varphi$ admits a $\mathcal{C}^{1}$-extension $\Vphi$ defined on $\R^n$ satisfying  equation (\ref{eq0}), and for some $K>0$,  
\begin{itemize}[leftmargin=*]
\item [] \hspace*{-0.5cm} $\lvert\Vphi(x)\rvert\leq K$, $\|\nabla\Vphi(x)\|\leq K$, and  $\|\nabla\Vphi(x)-\nabla\Vphi(y)\|\leq K\|x-y\|,\;\,\forall  x,\,y\in\R^n.$  
\end{itemize}
Employing (\ref{eq0}),  $(D)$ is equivalently  phrased in terms of the normal cone to $C$ and the extension $\Vphi$ of $\varphi$, as follows 
$$(D) \begin{sqcases}\dot{x}(t)\in  f_{\Vphi}(x(t),u(t)) -N_C(x(t)),\;\;\hbox{a.e.}\;t\in[0,1],\\x(0)\in C_0\subset C,	 \end{sqcases}\\$$
where $f_{\Vphi}\colon \R^n\times \R^m \f \R^n$ is defined by
\begin{equation} \label{f.phi} f_{\Vphi}(x,u):=f(x,u)-\nabla\Vphi(x),\;\;\;\forall (x,u)\in \R^n\times\R^m.\end{equation}
\end{enumerate}
\end{remark}

Therefore, we will interchangeably use throughout this paper  the original formulation of $(D)$ given in terms of  $\partial\varphi$ and $f$,  and  its reformulation provided in Remark \ref{connectednewD}  in terms of $N_C(\cdot)$ and  the function $f_{\Vphi}$. Note that assumptions (A1)-(A3) imply that,   for $\bar{M}:=M+K$, $f_{\Vphi}$ satisfies the following properties:
\begin{enumerate}[leftmargin=1.24cm, label=\textbf{(A\arabic*)$_\Vphi$}:]
\item The function $f_\Vphi$ is $\bar{M}$-Lipschitz on $C\times (\mathbb{U}+\tilde{\rho}\bar{B})$ with 
 $\|f_\Vphi(x,u)\| \leq \bar{M}$  for all $(x,u)\in C\times (\mathbb{U}+\tilde{\rho}\bar{B})$. 
\end{enumerate}
We define $\mathscr{U}$ to be
\begin{equation*}\label{scriptU} \mathscr{U}:=\{u:[0,1]\to \R^m : u \;\hbox{is measurable and}\; u(t)\in U(t), \; t\in [0,1] \;\hbox{a.e.}\}.\end{equation*}
\begin{remark}\label{lip.D} Using \cite[Lemma 4.3]{verachadi}, it is easy to see that  the assumptions (A1)-(A3) and  the boundedness of $C$ by some $M_C>0$ yield that any solution $x$ of ($D$) corresponding to $(x_0,u)\in C_0\times \mathscr{U}$ satisfies\begin{equation}\label{lipx} x(t)\in C,\;\; \forall t\in[0,1];\;\; \|x\|_{\infty}\le M_C;\;\;\hbox{and} \;\;\|\dot{x}\|_{\infty}\le 2 \bar{M}.
\end{equation} 	
\end{remark}

For given  $x(\cdot)\colon [0,1]\to \R^n,$ we use the following notations throughout this paper: $ I^0(x):=\{t\in [0,1]:  x(t)\in \bdry C\}$ and $I^{\-}(x):=[0,1]\setminus I^0(x)$.
 
The next result  characterizes  the solutions of ($D$) in terms of the solutions of a standard control system   containing an {\it extra} control $\xi$ that satisfies the mixed control-state {\it degenerate} constraint, $\xi(t)\psi(x(t))=0$.   The sufficiency part  is straightforward and was used  in \cite{verachadi}, while the necessary part follows   from applying  Filippov selection theorem (\cite[Theorem 2.3.13]{vinter}).
\begin{lemma}\label{characterizationD} Assume that \textnormal{(A1-(A3)} hold. Let  $u\in \mathscr{U}$ and   $x\in AC([0,1];\R^n)$ with  $x(0)\in C_0$ and $x(t)\in C$ for all $t\in [0,1]$. Then, $x$  is a solution for  $(D)$ corresponding  to the control  $u$   {\it if and only if} there exists  a nonnegative measurable function $\xi$   supported on $ I^{0}(x)$  such that  $(x,u,\xi)$ satisfies
\begin{equation}\label{admissible-P} \dot{x}(t)= f_\Vphi(x(t),u(t))-\xi(t) \nabla\psi(x(t)),\;\;\;t\in[0,1]\; \textnormal{a.e.} \end{equation}
In this case, the nonnegative function  $ \xi$ supported in $I^{0}(x)$ with $(x,u,\xi)$ satisfying   equation \eqref{admissible-P}, is unique,  belongs to  $L^{\infty}([0,1];\R^{+})$,  and 
\begin{equation}\label{boundxi}  \begin{cases} \xi(t)=0 &\;\;\hbox{for}\;\;t\in I^\-(x),\\[4pt]\xi(t)=\frac{\|\dot{x}(t)-f_{\Vphi}(x(t),u(t))\|}{\|\nabla\psi(x(t))\|}\in \left[0, \frac{\bar{M}}{2\eta}\right]&\;\;\hbox{for}\;\;t\in I^0(x) \;\textnormal{a.e.},\\[5pt] \|\xi\|_{\infty}\le \frac{\bar M}{2\eta}.	
\end{cases}
\end{equation}
\end{lemma}

Throughout the paper we shall employ the following notations, where $\eta$ and $\bar{M}$  are the constants given in (A2.2) and  (A1)$_\Vphi$, respectively.
\begin{itemize}[leftmargin=*]
\item $(\gk)_k$ is a sequence satisfying \begin{equation}  \label{assumpgk} \gk>\frac{2\bar{M}}{\eta}
 \;\,\hbox{for all}\; k\in\N,\;\hbox{and}\; \gk \xrightarrow[k\f\infty ]{}\infty.\end{equation}
 \item The sequence $(\a_k)_k $ is defined by  \begin{equation} \label{velo2} \a_{k}:=\frac{\ln \left(\frac{\eta\gk}{2\bar{M}}\right)}{\gk},\;\;\;k\in\N \end{equation}
 By \eqref{assumpgk} and \eqref{velo2}, we have that  \begin{equation}\label{velo2bis}  \gk e^{-\a_{k}\gk}=\frac{2\bar{M}}{\eta},\;\;\a_k>0,\;\; \a_k\searrow \;\hbox{and}\;\lim_{k\f\infty}\a_{k}=0.
\end{equation}
\item The sequence  $(\rho_k)_k$ is defined by $\rho_{k}:=\frac{\a_k}{\eta}$ for all $k\in\N$. By (\ref{velo2bis}) we have that  $ \rho_k>0$ for all $k\in\N$, $\rho_k\searrow\,$ and $\,\lim\limits_{k\f\infty}\rho_{k}=0.$
\item For $k\in\N$, we define the set \begin{equation}\label{C(k)} C(k):=\{x\in C : \psi(x)\leq -\a_k\}.
\end{equation}
\end{itemize}
The system $(D_{\gk})$ is defined as
$$({D}_{\gk}) \begin{sqcases} \dot{x}(t)=f_\Vphi(x(t),u(t))-\gk e^{\gk\psi(x(t))} \nabla\psi(x(t))\;\;\textnormal{a.e.}\; t\in[0,1],\\ 
x(0)\in C. \end{sqcases} $$
An important property shown in  \cite{verachadiimp} is the  invariance of $C$ for the dynamic $(D_{\gk})$, see \cite[Lemma 4.1]{verachadiimp}.  This fact is behind disposing  of the state constraint in ($D_{\gk}$),  which represents a good  approximation for $(D)$ (see Theorem \ref{prop2} and Corollary \ref{interiorrem}).

\begin{lemma}[Invariance of $C$ and uniform convergence] \label{invariance} Let \textnormal{(A1)-(A3)} be satisfied. Then, for each $k$, the system $(D_{\gamma_k})$ with given $x(0)=c_{\gk}\in C$ and $u_{\gk}\in \mathscr{U}$, has a unique solution $x_{\gk} \in W^{1,2}([0,1];\R^n)$ such that  $x_{\gk}(t)\in C$ for all $t\in [0,1]$, and, for $\a_0>0$  a  bound of $(c_{\gk})_k$ we have 
\begin{equation} \label{eq2} \|x_{\gk}\|_{\infty} \leq \a_0+\sqrt{\bar{M}^2+2}\;\;\;\hbox{and}\;\;\;\int_0^1\|\dot{x}_{\gk}(t)\|^2\sp dt\leq \bar{M}^2+2.\end{equation}
Hence, being equicontinuous and uniformly bounded, $(x_{\gk})_k$  admits a subsequence that converges uniformly to some $x \in W^{1,2}([0,1];\R^n)$ whose values are in $C$ and whose derivative $\dot{x}_{\gk}$ converges weakly in $L^2$ to $\dot{x}$. 
\end{lemma}

The properties of the sets $C(k)$ and  the role of the sequence $(\rho_k)_k$ are estblished in \cite[Theorem 3.1 and Remark 3.6]{verachadiimp}. We enlist here the items that deem important for this paper  when constructing the initial constraint set for the approximating problems $({P}_{\gk})$ and $(P_{\gk})$.

\begin{theorem}\textnormal{{\cite{verachadiimp}}}\label{propck} Under \textnormal{(A2.1)-(A2.3)}, the following assertions hold\sp$:$
\begin{enumerate}[$(i)$,leftmargin=0.8cm] 
\item For all $k$, the set $C(k)\subset \inte C$ and is compact, and, for $k$ sufficiently large, 
\begin{itemize}[leftmargin=0.2cm]
\item  $\bdry C(k)=\{x\in\R^n : \psi(x)=-\a_k\}\;\,\hbox{and}\;\,\inte C=\{x\in\R^n : \psi(x)<-\a_k\};$
\item  $(C(k))_k$ is a nondecreasing sequence whose  Painlev\'e-Kuratowski limit is $C$. \end{itemize}

\item There exist $r_{\bbp o}>0$ and $\bar{k}\in\N$ such that \begin{equation}\label{lastideahope}\left[C \cap \bar{B}_{r_{\bbp o}}\bp ({c})\right]-\rho_k\frac{\nabla\psi(c)}{\|\nabla\psi(c)\|}\subset \inte C(k),\;\;\;\forall\sp k\ge \bar{k}\;\;\;\hbox{and}\;\;\;\forall c\in \bdry C.\end{equation}

\item For $c\in \inte C$, there exist $\hat{k}_c \in \N$ and $\hat{r}_{\bbbp c}>0$  satisfying  
\begin{equation}\label{cinte C}\bar{B}_{\hat{r}_{\bbbp c}}({c})\subset \inte C(\hat{k}_c)\subset \inte C(k), \;\; \forall k\ge \hat{k}_c.\end{equation}
\end{enumerate}
\end{theorem}

\begin{remark} \label{intelemma} 
 From Theorem \ref{propck}, it follows that  for any $c\in C$, there exists a sequence $({c}_k)_k$ such that, for $k$ large enough, 
${c}_k\in \inte C(k)$ and $c_k\f c$. Indeed, for $c\in \bdry C$, take $c_k:= c- \rho_k\frac{\nabla\psi({c})}{\|\nabla\psi({c})\|}$ for all $k$, and  for $c\in \inte C$, take $c_k= c$  for all $k$.
\end{remark} 

The following theorem will be used  repeatedly in this paper. It is a special case of  \cite[Theorem 4.1 \& Lemma 4.2]{verachadiimp}. It  provides  a sufficient condition for the uniform limit $x$ of the  solution $x_{\gk}$ of $(D_{\gk})$  to be a solution of ($D$),  and it connects  the multiplier function $\xi$  corresponding  to $x$, via Lemma \ref{characterizationD},  to the  positive continuous penalty multiplier $\xi_{\gk}$, associated with  $x_{\gk}$  and defined by
 \begin{equation} \label{defxi} \xi_{\gk}(\cdot):= \gk e^{\gk\psi(x_{\gk}(\cdot))}.\end{equation}
 
\begin{theorem}[$(D_{\gk})_k$ \& $\xi_{\gk}$ approximate $(D)$ \& $\xi$] \label{prop2} Assume that \textnormal{(A1)-(A4.1)} hold. Let   $x_{\gk}$  be the solution of   $(D_{\gk})$  corresponding to $(c_{\gk},u_{\gk})$, as in \textnormal{Lemma \ref{invariance}}, and $x\in W^{1,2}([0,1];\R^n)$ be its uniform limit.
Then, the following statements are valid$\,:$
\begin{enumerate}[$(i)$,leftmargin=0.8cm] 
 \item The sequence   $(\xi_{\gk})_k$ admits a subsequence, we do not relabel,  that converges weakly in $L^2$ to  a nonnegative function $\xi\in L^2$ supported on $I^{0}(x)$.
\item If  for some  ${u}\in \mathscr{U}$, the sequence $u_{\gk}(t)\xrightarrow[]{\textnormal{a.e}.\;t\,} u(t)$, then $x$ is the unique solution of $(D)$ corresponding to $(x_0,u)$, and  $(x,u, \xi)$ satisfies equations \eqref{admissible-P}-\eqref{boundxi}. 
In particular, $\xi\in L^{\infty}([0,1];\R^{+})$ and is supported on $I^0(x)$.
\end{enumerate}
\end{theorem}

 \begin{remark}\label{newproof} Note that   when establishing Theorem \ref{prop2}$(ii)$ in \cite{verachadiimp},  the  arguments used to prove that $(x,u,\xi)$ satisfies \eqref{admissible-P}  are independent of  having  $\xi_{\gk}$ defined through \eqref{defxi}, and hence, this proof is valid for $\xi_{\gk}$  being any sequence of  $L^2$-functions  converging weakly in $L^2$ to $\xi$. Therefore, we have that  $(x,u,\xi)$ satisfies \eqref{admissible-P}  whenever $(x_j,u_j,\xi_j)_j$  is a sequence solving \eqref{admissible-P} with $x_j$ converging uniformly to $x$, $u_j(t)$ converging pointwise a.e. to $u(t)$, and $\xi_j$ converging weakly in $L^2$ to $\xi$. 
\end{remark}
The following result  is extracted from \cite[Theorem 5.1]{verachadiimp}, in which more properties are derived. It reveals the significance of initiating in Theorem  \ref{prop2} the solutions $x_{\gk}$ of $(D_{\gk})$ from the subset $C(k)$, defined in (\ref{C(k)}). 

\begin{theorem}[$x_{\gk}\in {C(k)}$, $\dot{x}_{\gk}$ \& $\xi_{\gk}$ bounded] \label{propooufnew} Assume \textnormal{(A1)-(A4.1)} hold. Let $(c_{\gk})_k$ be a sequence such that $c_{\gk}\in C(k)$, for $k$ sufficiently large. Then there exists ${k}_{o}\in\N$ such that for all sequences $(u_{\gk})_k$ in $\mathscr{U}$ and for all $k\geq {k}_{o}$, the solution $x_{\gk}$ of $(D_{\gk})$ corresponding to $(c_{\gk},u_{\gk})$ satisfies\sp$:$
\begin{enumerate}[$(i)$,leftmargin=0.8cm] 
\item $x_{\gk}(t)\in C(k)\subset \inte C$ for all $t\in[0,1]$.
\item $ 0\le \xi_{\gk}(t)\le \frac{2\bar{M}}{\eta}$ for all $t\in [0,1]$.
\item $ \|\dot{x}_{\gk}(t)\|\leq \bar{M}+\frac{2\bar{M}\bar{M}_\psi}{\eta}$ \,for \textnormal{a.e.} $t\in [0,1]$.
\end{enumerate}
\end{theorem}

  The next result is a  simplified version  of \cite[Corollary 5.1]{verachadiimp}. It is the converse of Theorem \ref{prop2},  as it  confirms that any given solution of ($D$) is approximated by  a solution of ($D_{\gk}$) that remains in the {\it interior} of $C$
  and enjoys all the properties displayed in Theorem \ref{propooufnew}.   
  
\begin{corollary}[Solutions of ($D$) are approximated by   sequences in $ C(k)$] \label{interiorrem} Assume that \textnormal{(A1)-(A4.1)} are satisfied. Let  $\x$ be the solution of  $(D)$ corresponding to  $(\x(0),\u)\in C_0\times \mathscr{U}$. Consider  $(\bar{c}_{\gk})_k$  the sequence in \textnormal{Remark \ref{intelemma}} that converges to $c:=\x(0)$, and $\x_{\gk}$ the solution of $(D_{\gk})$ corresponding to $(\bar{c}_{\gk},\u)$. Then,  there exists  $\hat{k}_{o} \in \N$ such that $\x_{\gk}$ and its associated $\bar{\xi}_{\gk}$ via  \textnormal{(\ref{defxi})} satisfy the conclusions $(i)$-$(iii)$ of \textnormal{Theorem \ref{propooufnew}} for all $k\ge \hat{k}_{o}$, and the following  holds true$\sp:$ The sequence $\x_{\gk}$ admits a subsequence, we do not relabel, that converges uniformly to $\x$, the corresponding subsequence for $\bar{\xi}_{\gk}$  converges weakly  in $L^2$ to some $\bar{\xi}\in L^{\infty},$ and $(\x,\u,\bar{\xi})$ satisfies \eqref{admissible-P}-\eqref{boundxi}. That is, $\bar{\xi}$ is the unique function corresponding to $(\x,\u)$ via \textnormal{Lemma \ref{characterizationD}}.
\end{corollary}

\section{Main results} \label{mainresults}
 This section consists of the main results of this paper, namely, the strong approximation of $(D)$ by $(D_{\gk})$ whenever the control is  ${W}^{1,2}$ (Theorem \ref{prop3new}  and Corollary \ref{newstrongconv}), an existence theorem for an optimal solution of $(P)$ (Theorem \ref{exop}), a strong converging continuous approximation for $(P)$ (Theorem \ref{lasthopefullybis}), and nonsmooth necessary optimality conditions in the form of weak-Pontryagin-type maximum principle (Theorem \ref{mpw12bisw12}).

 \subsection{$(D_{\gk})$ strongly approximates $(D)$ with $W^{1,2}$-controls} The following theorem constitutes the backbone of this paper. It shows that, when the underlying  control space  is  $\mathscr{W}$ (defined in \eqref{setw}),  the velocities  $\dot{x}_{\gk}$ and the functions  $\xi_{\gk}$  corresponding to  the approximating sequence ${x}_{\gk}$  in  Theorem \ref{propooufnew}, converge   {\it strongly} in $L^2$ to, respectively, $\dot{x}$ and  $\xi$, the functions obtained in  Theorem \ref{prop2}.  The proof of this theorem is postponed to Section \ref{auxresults}.
 
 \begin{theorem}[Strong convergence of the velocity sequence $\dot{x}_{\gk}$] \label{prop3new} Let the assumptions \textnormal{(A1)-(A4.2)} be satisfied. Consider  a sequence  $x_{\gk}$ solving $(D_{\gk})$ for some $(c_{\gk},u_{\gk})$, where $c_{\gk}\in C$, $c_{\gk}\f x_0\in C_0$,  $u_{\gk}\in \mathscr{W}$, and  $(\|\dot{u}_{\gk}\|_2)_k$  is bounded. Denote by $(x,\xi)$ the pair in $W^{1,2}\times L^2$ obtained via \textnormal{Lemma  \ref{invariance}} and \textnormal{Theorem  \ref{prop2}}$(i)$ such that a subsequence $($not relabeled$\,)$ of $(x_{\gk},\xi_{\gk})$ has $x_{\gk} $  converging uniformly in the set $C$ to $x$   and $(\dot{x}_{\gk},\xi_{\gk})$  converging weakly in $L^2$ to $(\dot{x},\xi)$. Then, the following hold\sp$:$
\begin{enumerate}[$(i)$,leftmargin=0.8cm] 
\item There exist  a subsequence $($not relabeled$\,)$ of  $u_{\gk}$, and $u\in \mathscr{W}$ such that $(x_{\gk}, u_{\gk})$ converges uniformly to $(x,u)$, and $(\dot{x}_{\gk},\dot{u}_{\gk},\xi_{\gk})$ converges weakly in $L^2$ to $(\dot{x},\dot{u},\xi)$.  The function $x$ is the unique solution to $(D)$ corresponding to $(x_0,u)$, and  $(x, u,\xi)$ satisfies \textnormal{\eqref{admissible-P}-\eqref{boundxi}}. In particular, $\xi\in L^{\infty}$ and is supported on $I^0(x)$.
\item Assume that $c_{\gk}\in C(k)$, for $k$ large. Then, in addition to  the conclusions in \textnormal{Theorem \ref{propooufnew}}, the  following holds\sp$:$ The sequence $(\dot{x}_{\gk}, \xi_{\gk})$ is in  $ W^{1,2}([0,1];\R^n)\times W^{2,2}([0,1]; \R^{+})$, has uniform bounded variations, and admits a subsequence, not relabeled, that converges pointwise, and hence, strongly in $L^2$  to $(\dot{x},\xi)$, with $\dot{x}\in BV([0,1];\R^n)$ and  $\xi\in BV([0,1];\R^+)$. In this case, \eqref{admissible-P}-\eqref{boundxi} hold for all  $ t\in[0,1],$  and  $x_{\gk}\f x$ strongly in the norm topology of $W^{1,2}([0,1];\R^n).$
\end{enumerate}
\end{theorem}

Applying Theorem \ref{prop3new}$(ii)$ to $\bar{c}_{\gk}$,  $u_{\gk}:=\u$, $\x_{\gk}$, and $\bar{\xi}_{\gk}$, the function associated to $\x_{\gk}$ via \eqref{defxi}, we obtain the following corollary that shows how the results in Corollary \ref{interiorrem} are improved when $W^{1,2}$-controls are utilized.

\begin{corollary}[$(D_{\gk})_k$ strongly approximates $(D)$]  \label{newstrongconv}  If, in addition to the assumptions of \textnormal{Corollary \ref{interiorrem}}, we have that $\x$ solves $(D)$ for $\u\in \mathscr{W}$\textnormal{ (not only in $\mathscr{U}$)},  then  $\bar{\xi}_{\gk}$, therein,   converges pointwise to $\bar{\xi}\in BV([0,1];\R^+)$  with $\bar{\xi}$ satisfying \eqref{existencenew}, and   $\x_{\gk}$, therein,  converges to $\x$ strongly in the norm topology of $W^{1,2}([0,1];\R^n)$. Moreover, $(\x,\u,\bar{\xi})$ satisfies \eqref{admissible-P}-\eqref{boundxi} for all $t\in[0,1]$, and  $\dot{\x}\in BV([0,1];\R^n)$. 
\end{corollary}
 
 \subsection{Existence of optimal solution for $(P)$}
 
Parallel to \cite[Theorems 4.1]{cmo0,cmo2}, where a discretization technique is used, the following existence theorem of  an optimal solution for the problem $({P})$ is established based on Corollary \ref{newstrongconv}. 
 
\begin{theorem}[Existence of solution for $({P})$] \label{exop} Assume hypotheses \textnormal{(A1)-(A4.3)}, $g: \R^{n}\times\R^n\to \R\cup\{\infty\}$ is lower semicontinuous, and that a minimizing sequence $(x_j,u_j)$ for $({P})$ exists  such that $(\|\dot{u}_j\|_2)_j$ is bounded.  Suppose  that $({P})$ has at least one admissible pair $(y_o,v_o)$ with $(y_o(0), y_o(1))\in \dom g$, then the problem $({P})$ admits a global optimal solution $(\x,\u)$ such that, along a subsequence, we have  
 $$x_{j}\xrightarrow[{W^{1,2}([0,1]; \R^n)}]{\textnormal{strongly}}\x,\;\;{u}_{j}\xrightarrow[{C([0,1]; \R^m)}]{\textnormal{uniformly}}{\u},\;\;\hbox{and}\;\;\dot{u}_{j}\xrightarrow[{L^2([0,1]; \R^m)}]{\textnormal{weakly}}\dot{\u}. $$	
\end{theorem}
\begin{proof} Given that $({P})$ has an admissible pair $(y_o,v_o)$ with $(y_o(0), y_o(1))\in \dom g$, then $\inf_{(x,u)} ({P})<\infty.$ As  $g$ is lower semicontinuous and all admissible solutions of ($P$) satisfy $(x(0),x(1))\in C_0\times (C_1\cap C)$, which is  compact, we deduce that $\inf_{(x,u)} ({P})$ is finite. On the other hand, being admissible for $({P})$, the minimizing sequence $(x_{j},u_{j})_j$ satisfies $(D)$ with $x_j(1)\in C_1$.  Hence, using that the sequence $(\|\dot{u}_j\|_2)_j$ is bounded,   Lemma \ref{compact} implies  the existence of $(\x,\u)\in W^{1,\infty}([0,1];\R^n)\times \mathscr{W}$ satisfying  $(D)$ and   $\x(1)\in C_1$,  with  $(x_j,u_j)$  converges uniformly to $(\x,\u)$, $(\dot{x}_j)_j$  converges {\it strongly} in $L^2$ to $\dot{\x}\in BV([0,1];\R^n)$, and $\dot{u}_j$ converges weakly in $L^2$ to $\dot{\u}$. Thus,  $(\x,\u) $ is admissible for $({P})$. Owed to the lower semicontinuity of $g$  and  to $(\x, \u)$ being the uniform limit of the minimizing sequence $(x_{j},u_{j})_j$, the optimality of the pair $(\x,\u)$  for the problem $({P})$  follows readily. \end{proof}

  \subsection{Continuous approximation for $(P)$}
  
  On the journey of seeking  for an optimal process ($\x,\u)$ of $(P)$ a  {\it continuous} approximations  consisting of {\it optimal} solutions for   properly-designed   standard control  problems,  it is important that  the convergence    to $(\x,\u)$  be {\it strong} in the norm topology of the considered space, namely,   the space $W^{1,2}([0,1];\R^n)\times \mathscr{W}$. Corollary \ref{newstrongconv} already answered this question for the $W^{1,2}$-strong approximation of a  solution $(\x,\u)$ of $(D)$ by  solutions of $(D_{\gk})$, in which the {\it same} control $\u$ is used. However, $\u$ may not necessarily be optimal for   approximating  optimal control problems over $(D_{\gk})$.
    
In this subsection, we approximate the problem $({P})$ by a certain sequence of optimal control problems  over $(D_{\gk})$ with  {\it special} initial and final state endpoints constraints ($C_0(k)\subset C(k)$ and $C_1(k)$ in a band around $C_1$), and with an objective function  particularly crafted so that an optimal control, $u_{\gk}$,  exists and has  $( \|\dot{u}_{\gk}\|_2)_k$  uniformly bounded, and hence,   the {\it strong} convergence of the optimal state velocities shall be deduced from Theorem \ref{prop3new}.  The  necessary optimality conditions for  $(P)$  are then established by taking the limit of the optimality conditions  for the corresponding approximating  problem.

For given $\delta>0$ and  $z\in C([0,1]; \R^s)$, we define the projection on $\R^s$ of the closed $\delta$-{\it tube}  around  $z$  by $\bar{\mathbb{B}}_{\delta}(z):=\bigcup\limits_{t\in [0,1]} \bar{B}_{\delta}(z(t)).$

Let $(\x,\u)\in W^{1,2}([0,1];\R^n)\times \mathscr{W}$ be a $W^{1,2}$-local minimizer for ($P$)  with associated $\delta.$ We fix $\delta_{\bbbp o}>0$ such that \begin{equation*}\label{delta0def} \delta_{\bbbp o}\leq \begin{cases}\min\{\hat{r}_{\x(0)},\delta\}&\qquad\;\hbox{if}\;\x(0)\in\inte C,\vspace{0.1cm}\\ \min\{r_{\bbp o},\delta\}&\qquad\;\hbox{if}\;\x(0)\in\bdry C,\end{cases} \end{equation*}
where $r_{\bbp o}>0$ is the constant in Theorem \ref{propck}$(ii)$, and  $\hat{r}_{\x(0)}>0$ with  $\hat{k}_{\x(0)}\in\N$  are the constants in   Theorem \ref{propck}$(iii)$ corresponding to  $c:=\x(0)$.

In the  remaining part below, we will assume that $f$ satisfies the following {\it local} version of (A1): 
\[\label{star}\exists\,\tilde{\rho}>0\;\hbox{such that}\; f\;\hbox{is Lipschitz on}\; [C\cap \bar{\mathbb{B}}_{\delta}(\x)]\times [(\mathbb{U}+\tilde{\rho}\bar{B}) \cap \bar{\mathbb{B}}_{\delta}(\u)].\tag{$*$}\]
Note that under the assumption \eqref{star}, the function $f$ can be extended to a {\it globally} Lipschitz function $\tilde{f}\colon\R^n\times\R^m\f\R$ by applying \cite[Theorem 1]{hiriartlip} to each component of $f$. Since in the rest of this section we only consider local optimality notions,  then, without loss of generality,   we shall use  the function $f$ instead of $\tilde{f}$.  Hence, when in this section $f$ is assumed to satisfy \eqref{star}, it is implied that  $f$ also satisfies  assumption (A1).

We proceed to suitably-formulate  a sequence of approximating problems $(P_{\gk})$ and show that its  {\it optimal solutions}  {\it strongly} converges in $W^{1,2}([0,1];\R^n)\times \mathscr{W}$ to the $W^{1,2}$-local minimizer $(\x,\u)$ of $({P})$. This naturally requires  the domain of  the approximating problem $(P_{\gk})$ to  be in $W^{1,2}([0,1];\R^n)\times \mathscr{W}$. The  initial state constraint  is taken  to be  \sloppy  $x(0)\in C_0(k)$, where $C_0(k)$ is the sequence of sets defined by \begin{equation}\label{c0kdef} \bp\bp\bp {C_0}(k):= 
     \begin{cases}C_0\cap \bar{B}_{\delta_{\bbbp o}}\bp\left({\x(0)}\right) ,\;\forall k\in\N, &\;\hbox{if}\;\x(0)\in\inte C,\vspace{0.1cm}\\ \left[C_0\cap\bar{B}_{\delta_{\bbbp o}}\bp\left({\x(0)}\right)\right]-\rho_k\frac{\nabla\psi(\x(0))}{\|\nabla\psi(\x(0))\|},\;\forall k\in\N,&\;\hbox{if}\; \x(0)\in\bdry C.\end{cases}\end{equation} 
and  the final state constraint  is   $x(1)\in C_1(k)$, where \begin{equation*} \label{C1kdefinition}  C_1(k):=\left[\left(C_1\cap\bar{B}_{\delta_{\bbbp o}}\bp({\x(1)})\right) -\x(1)+{\x}_{\gk}(1)\right]\cap C,\;\;\;k\in\N,\end{equation*} in which $\x_{\gk}$ is the solution of ($D_{\gk}$) corresponding to ($\bar{c}_{\gk}, \u$), where  $\bar{c}_{\gk}$ in $C_0(k)\cap \inte C(k)$, for $k$ large, and is defined via Remark \ref{intelemma}  for  $c:=\x(0)$, that is, 
\begin{equation*}\label{cbark} \bar{c}_k:= 
     \begin{cases}\x(0),\;\forall k\in\N, &\;\;\hbox{if}\;\x(0)\in\inte C,\vspace{0.1cm}\\ {\x(0)}-\rho_k\frac{\nabla\psi(\x(0))}{\|\nabla\psi(\x(0))\|},\;\forall k\in\N,&\;\;\hbox{if}\; \x(0)\in\bdry C.\end{cases}\end{equation*}
  Note that $C_0(k)$ and $C_1(k)$ are {\it closed}, for $k\in \N$.  On the other hand, as $\rho_{\gk}\f 0$, we have $\bar{c}_{\gk}\f \x(0)$, Corollary \ref{interiorrem} yields that  the sequence  $\x_{\gk}$ converges  in $C$ uniformly to $\x$, and hence,   $\x_{\gk}(1)\f\x(1)$.  Add to this that in $C_0(k)$, $\rho_k\f0$,  then,  for a fixed $\tilde{\rho}>0$, we have  that, for $k$ sufficiently large, \begin{equation}\label{c01kin} C_i(k) \subset \underbrace{\big[\left(C_i\cap \bar{B}_{\delta}(\x(i))\right)+\tilde{\rho}\bar{B}\big]\cap C}_{\tilde{C}_i(\delta)},\;\; \text{for}\;\;  i=0,1,\end{equation}
 \begin{equation}\label{limCi}\hspace{-.1 in}  \text{and}\;\,\lim_{k\to\infty} C_0(k) =C_0\cap\bar{B}_{\delta_{\bbbp o}}\bp({\x(0)})\;\, {\&}\;\, \lim_{k\to\infty} C_1(k) =C\cap C_1\cap\bar{B}_{\delta_{\bbbp o}}\bp({\x(1)}).\end{equation}
 \begin{remark} \label{c0subc} Notice that we can show that,  for $k$ large enough, we have  $C_0(k)\subset C(k)$, and hence,  by Theorem \ref{propooufnew}, any solution of $(D_{\gk})$ corresponding to $(c_{\gk},u_{\gk})$ with $c_{\gk}\in C_0(k)$ and $u_{\gk}\in \mathscr{U}$, satisfies the conditions $(i)$-$(iii)$ of this theorem. Indeed:
\begin{itemize}[leftmargin=*]
\item For  $\x(0)\in \inte C$,  use that  $\delta_{\bbbp o}\leq \hat{r}_{\x(0)}$ and  \eqref{cinte C} we get \begin{equation*}\label{case inteC}
  \bar{B}_{\hat{r}_{\x(0)}}\bp\left(\x(0)\right)\subset\inte C(k)\subset C(k),\quad \forall\sp k\ge \hat{k}_{\x(0)}.
  \end{equation*}
This gives that   
\begin{equation*} C_0(k):=C_0\cap  \bar{B}_{\delta_{\bbbp o}}\bp\left(\x(0)\right)\subset  \bar{B}_{\hat{r}_{\x(0)}}\bp\left(\x(0)\right)\subset C(k),\quad \forall\sp k\ge \hat{k}_{\x(0)}.
\end{equation*}
\item  For  $\x(0)\in \bdry C,$  use  that $\delta_{\bbbp o}\leq r_{\bbp o}$ and $C_0(k)$ is the nonempty set defined by the second equation of (\ref{c0kdef}),  to get  via  (\ref{lastideahope}) that 
$$C_0(k)\subset \inte C(k)\subset C(k),\quad \forall\sp k\geq\bar{k}.$$
\end{itemize}
\end{remark}

\begin{remark}\label{normalc1c0} For $c\in C_0(k)$ and $d\in C_1(k)$, the evaluation of  the normal cones $N^L_{C_0(k)}(c)$ and $N^L_{C_1(k)}(d)$  in terms of $N^L_{C_0}$ and $N^L_{C_1}$, respectively, is obtained using the local property of the limiting normal cone as the following
\begin{equation} \label{normalC0} \hspace{-0.22cm} N^L_{C_0(k)}(c)= 
     \begin{cases}N^L_{C_0}(c), &\;\;\;\hbox{if}\;\x(0)\in\inte C,\;\hbox{and}\vspace{-0.05cm} \\ & \;\;\;c\in B_{\delta_{\bbbp o}}(\x(0)), \vspace{0.2cm}\\ N^L_{C_0}\left(c+\rho_k\frac{\nabla\psi(\x(0))}{\|\nabla\psi(\x(0))\|}\right)\bp, &\;\;\;\hbox{if}\; \x(0)\in\bdry C,\;\hbox{and}\vspace{-0.05cm} \\ & \;\;\left(c+\rho_k\frac{\nabla\psi(\x(0))}{\|\nabla\psi(\x(0))\|}\right)\in B_{\delta_{\bbbp o}}(\x(0)). \end{cases}\vspace{0.2cm}
 \end{equation}
\begin{equation} \label{normalC1} N^L_{C_1(k)}(d)=N^L_{C_1}(d+\x(1)-\x_{\gk}(1)),\;\;\forall d\in(\inte C)\cap B_{\delta_{\bbbp o}}(\x(1)) .\end{equation}
\end{remark}

We introduce the following sequence of approximating problems: 
 $$\begin{array}{l} ({P}_{\gk})\colon\; \textnormal{Minimize}\\[2pt] \hspace{0.5cm}{J}(x,y,z,u):= g(x(0),x(1))+\frac{1}{2}\left(\|u(0)-\u(0)\|^2 + z(1)+\|x(0)-\x(0)\|^2\right)\\[5pt] \hspace{0.5cm} \textnormal{over}\, (x,y,z,u)\in W^{1,2}([0,1];\R^n)\times AC([0,1];\R)\times AC([0,1];\R)\times \mathscr{W}\\ [2.6pt]\hspace{0.5cm} \textnormal{such that} \\[2.6pt] \hspace{1.25cm}
 \begin{cases} (\tilde{D}_{\gk}) \begin{sqcases}\dot{x}(t)={f}_\Vphi(x(t),u(t))-\gk e^{\gk\psi(x(t))} \nabla\psi(x(t)),\;\; t\in[0,1]\; \textnormal{a.e.},\\ \dot{y}(t)=\|\dot{x}(t)-\dot{\x}(t)\|^2,\;\;t\in[0,1]\; \hbox{a.e.},\\ \dot{z}(t)=\|\dot{u}(t)-\dot{\u}(t)\|^2,\;\;t\in[0,1]\;\hbox{a.e.},\\ (x(0),y(0),z(0))\in {C}_0(k)\times\{0\}\times\{0\}, \end{sqcases}\vspace{0.1cm}\\x(t)\in \bar{B}_{\delta}(\x(t))\;\hbox{and}\; u(t)\in U(t)\cap \bar{B}_{\delta}(\u(t)),\;\;\forall t\in[0,1],\vspace{0.1cm}\\(x(1),y(1),z(1))\in {C}_1(k)\times [-\delta,\delta]\times[-\delta,\delta].\end{cases} \end{array}$$
 Note that Lemma \ref{invariance} and the constraints on $u(\cdot)$ confirm that $({P}_{\gk})_{k}$ is actually equivalent to having therein $(x,u)\in AC([0,1];\R^n)\times AC([0,1];\R^m).$
 
Now we are ready to state our {\it continuous} approximation result,  which is parallel to the corresponding result in \cite{cmo0,cmo,cmo2,chhm}, where discrete approximations are used. The proof of this approximation result is presented in Section \ref{auxresults}.

\begin{theorem}[$({P}_{\gk})$ approximates $({{P}})$] \label{lasthopefullybis} Let $(\x,\u)$ be a $W^{1,2}$-local minimizer $({P})$ with associated $\bar{\xi}\in L^{\infty}$ via \textnormal{Lemma \ref{characterizationD}}. Assume that  \textnormal{(A2)-(A4.3)} hold, $g$ is continuous on $\tilde{C}_0(\delta)\times \tilde{C_1}(\delta)$, and for some $\tilde{\rho}>0$, $f$ is Lipschitz on $[C\cap \bar{\mathbb{B}}_{\delta}(\x)]\times [(\mathbb{U}+\tilde{\rho}\bar{B}) \cap \bar{\mathbb{B}}_{\delta}(\u)]$. Then for $k$ sufficiently large, the problem $({P}_{\gk})$ has an optimal solution $(x_{\gk},y_{\gk},z_{\gk},u_{\gk})$ such that, for  $\xi_{\gk}$ defined in \eqref{defxi}, we have, along a subsequence, we do not relabel, that 
$$(x_{\gk},u_{\gk})\xrightarrow[{W^{1,2}\times\mathscr{W}}]{\textnormal{strongly}}(\x,\u),\;\,(y_{\gk},z_{\gk})\xrightarrow[{W^{1,1}([0,1];\R^+\times\R^+)}]{\textnormal{strongly}} (0,0),\;\,\xi_{\gk}\xrightarrow[{L^2([0,1]; \R^+)}]{\textnormal{strongly}}\bar{\xi},$$ all the conclusions of \textnormal{Theorem \ref{propooufnew}} hold, including that $x_{\gk}(t)\in\inte C$ for all $t\in[0,1]$, and  for all $k$ sufficiently large, 
    $$x_{\gk}(i)\in \big[\left(C_i\cap \bar{B}_{\delta_{\bbbp o}}(\x(i))\right)+\tilde{\rho}{B}\big]\cap(\inte C)\subset \inte \tilde{C}_i(\delta), \quad \text{for}\;\; i=0,1.$$
Moreover, $\dot{\x}\in BV([0,1];\R^n)$, $\bar{\xi}\in BV([0,1];\R^+)$, and \eqref{admissible-P}-\eqref{boundxi} are valid at $(\x,\u,\xxi)$ for all $t\in[0,1]$.\end{theorem} 

We proceed to rewrite the problems $({P}_{\gk})$ as an optimal control problem with state constraints.  Given  $(\x,\u)\in W^{1,2}([0,1];\R^n)\times \mathscr{W}$  a $W^{1,2}$-local minimizer for $({P})$,  for  $\bar{v}:=\dot{\u}$,    
 $({P}_{\gk})$ is reformulated in the following way:
$$ \begin{array}{l} ({P}_{\gk})\colon\; \hbox{Minimize}\;\\[2pt]  \hspace{0.1cm} g(x(0),x(1))+\frac{1}{2}\left(\|u(0)-\u(0)\|^2 + z(1) +\|x(0)-\x(0)\|^2\right)\hbox{over}\;\\[2pt] \hspace{0.1cm} (x,y,z,u)\in AC([0,1];\R^n)\bp\times\bp AC([0,1];\R)\bp\times\bp AC([0,1];\R)\bp\times\bp AC([0,1];\R^m)\\[1pt] \hspace{0cm} \;\hbox{and measurable functions}\; v\colon[0,1]\f\R^m\;\hbox{such that}\\ \hspace{.5cm} 
\begin{cases} \begin{sqcases} \dot{x}(t)={f}_\Vphi(x(t),u(t))-\gk e^{\gk\psi(x(t))} \nabla\psi(x(t)),\;\, t\in[0,1]\;\textnormal{a.e.,}\vspace{0.07cm}\\
\dot{u}(t)=v(t),\;\, t\in[0,1]\,\textnormal{a.e.},\vspace{0.07cm}\\
\dot{y}(t)=\|{f}_\Vphi(x(t),u(t))-\gk e^{\gk\psi(x(t))} \nabla\psi(x(t))-\dot{\x}(t)\|^2,\;\, t\in[0,1]\,\textnormal{a.e.},\vspace{0.07cm}\\\dot{z}(t)=\|v(t)-\bar{v}(t)\|^2,\;\, t\in[0,1]\;\textnormal{a.e.},\end{sqcases}\vspace{0.07cm}\\x(t)\in \bar{B}_{\delta}(\x(t))\;\hbox{and}\; u(t)\in U(t)\cap \bar{B}_{\delta}(\u(t)),\;\;\forall t\in[0,1],\\[1pt] (x(0),u(0),y(0),z(0))\in C_0(k)\times \R^m\times\{0\}\times\{0\},\\[1pt](x(1),u(1),y(1),z(1))\in C_1(k)\times \R^m\times [-\delta,\delta]\times [-\delta,\delta].\end{cases}
 \end{array}$$

In the following proposition we apply to the above sequence of reformulated problems $({P}_{\gk})$, the nonsmooth Pontryagin maximum principle for optimal control problems with {\it multiple} state constraints (see e.g., \cite[page 331]{vinter} and \cite[p.332]{vinter}). For this purpose,  $(x,y,z,u)$ is  the state function in $(P_{\gk})$ and $v$ is the control. Thus,  $(x_{\gk},y_{\gk}, z_{\gk}, u_{\gk})$ is the optimal state, where $(x_{\gk},u_{\gk})$ is obtained from Theorem \ref{lasthopefullybis},  $y_{\gk}(t):=\int_0^t \|\dot{x}_{\gk}(s)-\dot{\x}(s)\|^2\;ds$, $z_{\gk}(t):=\int_0^t \|\dot{u}_{\gk}(s)-\dot{\u}(s)\|^2\;ds$, and $v_{\gk}=\dot{u}_{\gk}$ is the optimal control. Hence, the function ${f}(\cdot,\cdot)$ is required to be Lipschitz {\it near}   $(x_{\gk},u_{\gk})$, which follows from \eqref{star}, since  $x_{\gk}(t)\in \inte C$  and $(x_{\gk},u_{\gk})$ converges uniformly to $(\x,\u)$ (see Theorem \ref{lasthopefullybis}).
Furthermore,  as the objective function $g$ must be Lipschitz  {\it near}  $(x_{\gk}(0),x_{\gk}(1))$,  we introduce the following local assumption on $g$ in which $\tilde{C}_0(\delta)$ and $\tilde{C}_1(\delta)$ are defined in \eqref{c01kin}:
\vspace{.05 in}
\begin{center} $\exists\,\tilde{\rho}>0$ such that $g$ is Lipschitz on $\tilde{C}_0(\delta)\times \tilde{C_1}(\delta)$.\end{center}
\vspace{.05 in}
On the other hand, the following {\it constraint qualification} property (CQ) is required. For a given multifunction $F\colon [0,1]\rightrightarrows\R^m$, with nonempty and closed values, and for $h\in C([0,1];F)$, that is, $h\in C([0,1];\R^m)$  and satisfies $h(t)\in F(t)$ for all $t\in [0,1]$, we say that $F(\cdot)$ satisfies {\it the constraint qualification  at $h$} if
\vspace{.05 in}
\begin{center}\hspace{-1.4 in}{\bf(CQ)}\qquad  $\conv (\bar{N}^L_{F(t)}(h(t)))$ is pointed for all $t\in[0,1]$.\end{center}
\vspace{.05 in}
Here, $ \bar{N}^L_{F(t)}(y)$ stands for the graphical closure at $(t,y)$ of the multifunction $(t,y)\mapsto{N}^L_{F(t)}(y)$, that is, the graph of $\bar{N}^L_{F(\cdot)}(\cdot)$ is the closure of the graph of $N^L_{F(\cdot)}(\cdot)$. For more information about the (CQ) property, see Remark \ref{CQremark}. 

It is worth noting that in  \cite{cmo0,cmo,cmo2}, where $W^{1,2}$-controls are employed, the control sets  $U(t)$ are assumed to be $\R^m$, for all $t\in [0,1]$, and hence, the (CQ) property is trivially satisfied.

\begin{proposition}[Maximum Principle for approximating problems $(P_{\gk})$] \label{mpw12aproxw12} Let $(\x,\u)$ be a $W^{1,2}$-local minimizer for $(P)$. Assume that {\textnormal{(A2)-(A4)}} hold, and for some $\tilde{\rho}>0$, $f$ is Lipschitz on $[C\cap \bar{\mathbb{B}}_{\delta}(\x)]\times [(\mathbb{U}+\tilde{\rho}\bar{B}) \cap \bar{\mathbb{B}}_{\delta}(\u)]$ and $g$ is Lipschitz on $\tilde{C}_0(\delta)\times \tilde{C_1}(\delta)$. Consider the optimal sequence $(x_{\gk},y_{\gk},z_{\gk},u_{\gk},v_{\gk})$ for $({P}_{\gk})$ obtained via \textnormal{Theorem \ref{lasthopefullybis}}. If for $k$ sufficiently large, $U(\cdot)$ satisfies the constraint qualification \textnormal{(CQ)} at $u_{\gk}$, then for $k$ large enough, there exist $\l_{\gk}\geq 0$, $p_{\gk}\in AC([0,1];\R^n)$, $q_{\gk}\in AC([0,1];\R^m)$,  $\Omega_{\gk} \in NBV([0,1];\R^m)$, $\mu^{\bbbp o}_{\gk}\in C^{\oplus}([0,1];\R^{m})$, and a $\mu_{\gk}^{\bbbp o}$-integrable function $\beta_{\gk}\colon [0,1]\f\R^m$  such that  $\Omega_{\gk}(t)= \int_{[0,t]} \beta_{\gk}(s)\mu^{\bbbp o}_{\gk}(ds)$, for all $t\in(0,1]$, and\sp$:$
\begin{enumerate}[$(i)$,leftmargin=0.8cm] 
\item {\bf (The nontriviality condition)} For all $k\in\N$, we have $$\|p_{\gk}(1)\|+ \|q_{\gk}\|_\infty + \|\mu^{\bbbp o}_{\gk}\|\TV + \l_{\gk}=1;$$
\item {\bf (The adjoint equation)} For \textnormal{a.e.} $t\in [0,1],$ 
\begin{eqnarray}\nonumber\left(
\begin{array}{c}
\dot{p}_{\gk}(t)\vspace{0.2cm}\\
\dot{q}_{\gk}(t)\\
\end{array}
\right)\in &-&\left(\partial^{\sp (x,u)} {f}_{\Vphi}(t,x_{\gk}(t),u_{\gk}(t))\right)\tran p_{\gk}(t)\\ &+& \label{adjapp} \left(
\begin{array}{c}
\gk e^{\gk \psi(x_{\gk}(t))}\partial^2\psi(x_{\gk}(t))p_{\gk}(t)\vspace{0.2cm}\\
0\\
\end{array}\right)\\ &+& \left(
\begin{array}{c}
\gk^2 e^{\gk \psi(x_{\gk}(t))}\nabla\psi(x_{\gk}(t))\<\nabla\psi(x_{\gk}(t)),p_{\gk}(t)\>\vspace{0.2cm}\\
0\\
\end{array}
\right);\nonumber\end{eqnarray}
\item {\bf (The transversality equation)} 
$$(p_{\gk}(0),-p_{\gk}(1))\in $$
\begin{equation*}\hspace{-0.4cm}\l_{\gk}\partial^L g(x_{\gk}(0),x_{\gk}(1))\bbp+\bbp\big[\big(\l_{\gk}(x_{\gk}(0)\bbp-\bbp\x(0)) \bbp+\bbp N_{C_0(k)}^L(x_{\gk}(0))\big) \bbp \times \bbp N_{C_1(k)}^L(x_{\gk}(1))\big],\vspace{0.2cm}\end{equation*} 
 \begin{equation*} \hbox{and}\;\;q_{\gk}(0)=\l_{\gk}(u_{\gk}(0)-\u(0)),\;\;\;-q_{\gk}(1)=\Omega_{\gk}(1); \end{equation*}
\item {\bf (The maximization condition)} For \textnormal{a.e.} $t\in [0,1],$ \begin{equation*} \max_{v\in\R^m} \left\{\<q_{\gk}(t) +\Omega_{\gk}(t),v\>-\frac{\l_{\gk}}{2}\|v-\dot{\bar{u}}(t)\|^2\right\}  \;\hbox{is attained at}\;  \dot{u}_{\gk}(t);\end{equation*}   
\item {\bf (The measure properties)}
$$\supp\{\mu_{\gk}^{\bbbp o}\}\subset  \left\{t\in [0,1] : (t,u_{\gk}(t))\in \bdry\Gr\left[U(t)\cap \bar{B}_{\delta}(\u(t))\right]\right\},\; \mbox{and} $$ 
$$\beta_{\gk}(t)\in \partial_u^{\sp>} d(u_{\gk}(t),U(t)\cap \bar{B}_{\delta}(\u(t))) \;\;\;\mu_{\gk}^{\bbbp o}\;\textnormal{a.e.,}\;\hbox{with}$$
 $$\partial_u^{\sp>} d(u_{\gk}(t),U(t)\cap \bar{B}_{\delta}(\u(t))) \subset \left[\conv\bar{N}^L_{U(t)\cap \bar{B}_{\delta}(\u(t))}(u_{\gk}(t))\cap \left(\bar{B}\setminus\{0\}\right)\right].$$ 
\end{enumerate}
\end{proposition}

 \subsection{Necessary optimality conditions for $(P)$}
The main result of this subsection is the following theorem which provides necessary optimality conditions for the $W^{1,2}$-local minimizer, $(\x,\u)$,  of  $({P})$. 

The following notations are used in the statement of the theorem: 
\begin{itemize}[leftmargin=*]
\item $\partial_\ell\varphi$ and $\partial^2_\ell\varphi$ stand, respectively, for the {\it extended  Clarke generalized gradient} and the {\it extended  Clarke generalized Hessian} of $\varphi$ defined on $C$ via \textnormal{(\ref{phil1})}\,\&\,\textnormal{(\ref{phil2})}.
\item $\partial^{\sp (x,u)}_\ell f (\cdot,\cdot)$ is the {\it extended  Clarke generalized Jacobian} of $f(\cdot,\cdot)$ defined on     
 $\left[C \cap \bar{B}_{\delta}(\x(t))\right]\times \big[(U(t)+\tilde{\rho}\bar{B})\cap  \bar{B}_{\delta}(\u(t))\big]$   via \eqref{phil3}.
\item $\partial^2_\ell\psi$ is the {\it Clarke generalized Hessian relative} to $\inte C$ of $\psi$, defined via \eqref{phil4}.
\item $\partial^L_\ell g$ is the {\it limiting subdifferential} of $g$ {\it relative} to  $\inte \big(\tilde{C}_0(\delta)\times \tilde{C}_1(\delta)\big)$,  defined  via \textnormal{(\ref{delelel})}.
\end{itemize}

\begin{theorem}[Necessary optimality conditions for $(P)$] \label{mpw12bisw12} Let $(\x,\u)$ be a $W^{1,2}$-local minimizer for $(P)$. Let $\xxi\in L^{\infty}([0,1];\R^{+})$ be the function supported on $I^0(\x)$ and associated  to $(\x,\u)$ via \textnormal{Lemma \ref{characterizationD}}. Assume that {\textnormal{(A2)-(A4)}} hold, $U(\cdot)$ satisfies the constraint qualification \textnormal{(CQ)} at $\u$, and for some $\tilde{\rho}>0$, $f$ is Lipschitz on $[C\cap \bar{\mathbb{B}}_{\delta}(\x)]\times [(\mathbb{U}+\tilde{\rho}\bar{B}) \cap \bar{\mathbb{B}}_{\delta}(\u)]$ and $g$ is Lipschitz on $\tilde{C}_0(\delta)\times \tilde{C_1}(\delta)$. Then $\dot{\x}\in BV([0,1];\R^n)$ and $\bar{\xi}\in BV([0,1];\R^+)$, and there exist $\l\geq 0$, an adjoint vector $\bar{p}\in BV([0,1];\R^n)$, a finite signed Radon measure $\bar{\nu}$ on $[0,1]$ supported on $I^{0}(\x)$, $L^{\infty}$-functions $\bar{\zeta}(\cdot)$, $\bar{\theta}(\cdot)$ and $\bar{\vartheta}(\cdot)$ in $\mathscr{M}_{n\times n}([0,1])$, an $L^{\infty}$-function $\bar{\omega}(\cdot)$ in $\mathscr{M}_{n\times m}([0,1])$, such that for $t\in [0,1]$ \textnormal{a.e.},
$$\left((\zzeta(t),\oomega(t)),\ttheta(t), \vvartheta(t)\right)\bbp\in\bbp \;\partial^{\sp (x,u)}_\ell f(\x(t),\u(t))\times \partial^2_\ell\varphi(\x(t))\times \partial^2_\ell\psi(\x(t)),$$
and the following hold\sp$:$
\begin{enumerate}[$(i)$,leftmargin=0.8cm] 
\item {\bf(The admissible equation)}
\begin{enumerate}[$(a)$]
\item $\dot{\x}(t)= f(\x(t),\u(t))- \nabla_{\bp\ell}\sp\sp\varphi(\x(t))-\bar{\xi}(t) \nabla\psi(\x(t)),\;\;  \forall t\in [0, 1],$
\item $\psi(\x(t))\le 0,\;\;\forall\sp t\in [0,1];$
\end{enumerate}
\item {\bf (The nontriviality condition)} $$\|\p(1)\| + \l=1;$$
\item {\bf (The adjoint equation)} 
 For any $h\in C([0,1];\R^n),$ we have\vspace{0.1cm} \begin{eqnarray*}\hspace*{-0.2cm} \int_{[0,1]}\<h(t),d\p(t)\>&= &\int_0^1 \left\<h(t),\left(\ttheta(t)-\zzeta(t)\tran\right) \p(t)\right\>\sp dt \\&+& \int_0^1 \xxi(t)\left\<h(t),\vvartheta(t)p(t)\right\>\sp dt+  \int_{[0,1]} \<h(t),\nabla \psi(\x(t))\>\sp d\nnu;\\[-0.3cm]\end{eqnarray*} 
 \item {\bf (The complementary slackness conditions)} 
\begin{enumerate}[$(a)$]
\item $\bar{\xi}(t)=0,\;\;\forall\sp t\in I^{\-}(\x)$,
\item $\bar{\xi}(t)\<\nabla\psi(\x(t),\bar{p}(t)\>=0,\;\;\forall\sp t\in [0, 1]\; \textnormal{a.e.};$
\end{enumerate}
\item {\bf (The transversality equation)} $$ (\p(0),-\p(1))\in  \l\partial_\ell^L g(\x(0),\x(1))+ \big[N_{C_0}^L(\x(0))\times\;N^L_{C_1}(\x(1))\big];$$
\item {\bf (The weak maximization condition)}  $$\oomega(t)\tran \p(t)\in \conv\bar{N}^L_{U(t)\cap \bar{B}_{\delta}(\u(t))}(\u(t)),\;\;  t\in [0,1]\; \textnormal{a.e.}\vspace{0.1cm}$$
If in addition there exist $\e_{\bbbp o}>0$ and $r>0$ such that $U(t)\cap \bar{B}_{\e_{\bbbp o}}(\u(t))$ is $r$-prox-regular for all $t\in [0,1]$, then we have
$$  \textstyle \max\left\{\left\<\oomega(t)\tran \p(t),u\right\>- \frac{\|\oomega(t)\tran \p(t)\|}{\min\{\e_{\bbbp o},2r\}}\|u-\u(t)\|^2 : u\in U(t)\right\} $$ is attained at $\bar{u}(t)$ for  $t\in [0,1]\,\textnormal{a.e.}$ 
\end{enumerate}
 \end{theorem}
 
 \begin{remark} \label{uconvex}  Condition $(vi)$ of  Theorem \ref{mpw12bisw12} admits simplified forms when $U(\cdot)$ possesses extra properties: 
\begin{itemize}[leftmargin=*]
\item If  $U(t)$ is $r$-{\it prox-regular} for all $t\in [0,1]$, then taking $\e_{\bbbp o}\f\infty$, the maximization condition $(v)$  reduces to 
$$\textstyle  \max\left\{\left\<\oomega(t)\tran \p(t),u\right\>- \frac{\|\oomega(t)\tran \p(t)\|}{2r}\|u-\u(t)\|^2 : u\in U(t)\right\} $$ is attained at $\bar{u}(t)$ for $t\in [0,1]\,\textnormal{a.e.}$ 
\item If  $U(t)\cap \bar{B}_{\e_{\bbbp o}}(\u(t))$ is {\it convex} for all $t\in [0,1]$, then taking $r\f\infty$, the maximization condition $(v)$  reduces to  
$$ \textstyle \max\left\{\left\<\oomega(t)\tran \p(t),u\right\>- \frac{\|\oomega(t)\tran \p(t)\|}{\e_{\bbbp o}}\|u-\u(t)\|^2 : u\in U(t)\right\} $$ is attained at $\bar{u}(t)$ for  $t\in [0,1]\,\textnormal{a.e.}$ 
\item If  $U(t)$ is {\it convex} for all $t\in [0,1]$, then taking both $\e_{\bbbp o}\f\infty$ and $r\f\infty$, the maximization condition $(v)$ reduces to   
$$\max\left\{\left\<\oomega(t)\tran \p(t),u\right\> : u\in U(t) \right\}\;\hbox{is attained at}\;\bar{u}(t)\;\hbox{for}\;t\in [0,1]\, \textnormal{a.e.}$$ 
\end{itemize}
\end{remark}
\begin{proof}\bp{\it of} Theorem \ref{mpw12bisw12}. Theorem \ref{lasthopefullybis} produces a subsequence of $(\gk)_k$, we do not relabel, and a corresponding sequence $(x_{\gk},y_{\gk}, z_{\gk},u_{\gk})_k$, with associated $(\xi_{\gk})_k$ defined via  \eqref{defxi},  such that 
\begin{itemize}[leftmargin=*]
\item  For each $k$, the quadruplet $(x_{\gk},y_{\gk}, z_{\gk},u_{\gk})$ is optimal for $(P_{\gk})$.
\item $(x_{\gk},u_{\gk})\hspace{-0.05cm}\xrightarrow[{W^{1,2}\times\mathscr{W}}]{\textnormal{strongly}}\hspace{-0.05cm}(\x,\u)$, $(y_{\gk},z_{\gk})\hspace{-0.05cm}\xrightarrow[{W^{1,1}([0,1];\R^+\times\R^+)}]{\textnormal{strongly}}\hspace{-0.05cm} (0,0)$, $\xi_{\gk}\hspace{-0.05cm}\xrightarrow[{L^2([0,1]; \R^+)}]{\textnormal{strongly}}\hspace{-0.05cm}\bar{\xi}.$
\item  $\dot{\x}\in BV([0,1];\R^n)$, $\bar{\xi}\in BV([0,1];\R^+)$, and \eqref{admissible-P}-\eqref{boundxi} are valid at $(\x,\u,\xxi)$ for all $t\in[0,1]$.
\item  All the conclusions of \textnormal{Theorem \ref{propooufnew}} hold, including $(x_{\gk})_k$ is uniformly Lipschitz and $x_{\gk}(t)\in\inte C$ for all $t\in[0,1]$.
\item  For all $k$, we have $$x_{\gk}(i)\in \big[\left(C_i\cap \bar{B}_{\delta_{\bbbp o}}(\x(i))\right)+\tilde{\rho}{B}\big]\cap(\inte C)\subset \inte \tilde{C}_i(\delta), \quad \text{for}\;\; i=0,1.$$

\end{itemize}
In order to apply  Proposition \ref{mpw12aproxw12}, we shall show that the constraint qualification (CQ) that holds for $U(\cdot)$ at $\u$,  also holds true at $u_{\gk}$, for $k$ large enough. Indeed, if this is false, then, by Remark \ref{CQremark}$(i)$, there exist an increasing sequence $(k_n)_n$ in $\N$ and a sequence $t_n\in [0,1]$ such that $t_n\f t_o\in [0,1]$ and \begin{equation} \label{CQ1} 0\in\partial_u^{\sp >}d_U(t_n,u_{\gkn}(t_n)),\;\;\forall n\in\N.\end{equation} The continuity of $\u$ and the uniform convergence of $u_{\gkn}$ to $\u$ yield that the sequence $(u_{\gkn}(t_n))_n$ converges to $\u(t_o)$. Hence, using that the multifunction $(t,x)\mapsto \partial_u^{\sp >}d_U(t,x)$ has closed values and a closed graph, we conclude from 
\eqref{CQ1}  that $0\in \partial_u^{\sp >}d_U(t_o,\u(t_o))$. This  contradicts that  the constraint qualification is satisfied by $U(\cdot)$ at $\u$. Thus,  for $k$ sufficiently large, $U(\cdot)$ satisfies the constraint qualification (CQ)  at $u_{\gk}$. 

 Hence, by Proposition \ref{mpw12aproxw12}, there exist a subsequence of $(\gk)_k$, we do not relabel, and corresponding sequences $p_{\gk}$, $q_{\gk}$ $\mu_{\gk}$ and $\l_{\gk}$  satisfying  conditions $(i)$-$(v)$ therein.

Using (\ref{adjapp}), \eqref{f.phi}, and  that for all $t\in [0,1]$ we have $$(x_{\gk}(t),u_{\gk}(t))\in\inte\bp\bbp \left[\left(C \cap \bar{B}_{\delta}(\x(t))\right)\bp\times\bp\big((U(t)\bbp+\bbp\tilde{\rho}\bar{B})\cap  \bar{B}_{\delta}(\u(t))\big)\right],$$  we obtain sequences $\zeta_{\gk},\,\theta_{\gk}$ and $\vartheta_{\gk}$ in $\mathscr{M}_{n\times n}([0,1])$ and $\omega_{\gk}$ in $\mathscr{M}_{n\times m}([0,1])$ 
such that, for a.e. $t\in [0,1],$
\begin{equation*} (\zeta_{\gk}(t),\omega_{\gk}(t))\in \partial^{\sp (x,u)}_\ell f(x_{\gk}(t),u_{\gk}(t)),\end{equation*}
\begin{equation*} (\theta_{\gk}(t),\vartheta_{\gk}(t))\in \partial^2_\ell\varphi(x_{\gk}(t))\times \partial^2_\ell\psi(x_{\gk}(t)), \end{equation*}
\begin{eqnarray}\label{mpim1} \dot{p}_{\gk}(t)&=& (\theta_{\gk}(t)- \zeta_{\gk}(t))\tran p_{\gk}(t)+ \gk e^{\gk \psi(x_{\gk}(t))}\vartheta_{\gk}(t)\,p_{\gk}(t)\\[3pt] &+& \gk^2 e^{\gk \psi(x_{\gk}(t))}\nabla\psi(x_{\gk}(t))\<\nabla\psi(x_{\gk}(t)),p_{\gk}(t)\>, \;\hbox{and}\nonumber
\end{eqnarray} 
\begin{equation}\label{w12mp1} \dot{q}_{\gk}(t)=-(\omega_{\gk}(t))\tran p_{\gk}(t).\end{equation} 
Note that for each $k$, the functions $p_{\gk}, \dot{p}_{\gk}, q_{\gk}, \dot{q}_{\gk}, x_{\gk}$, $u_{\gk}$ and $\u$ are measurable on $[0,1]$, and the multifunctions $\partial^{\sp (x,u)}_\ell f(\cdot,\cdot)$, $\partial_{\ell}^2\varphi(\cdot)$, and $\partial_{\ell}^2\psi(\cdot)$ are measurable and have closed graphs with nonempty, compact, and convex values. Using (A1), (A2.1), and (A3),  the Filippov measurable selection theorem (see \cite[Theorem 2.3.13]{vinter}) yields that we can assume the measurability of the functions $\zeta_{\gk}(\cdot)$, $\theta_{\gk}(\cdot)$, $\vartheta_{\gk}(\cdot)$ and $\omega_{\gk}(\cdot)$. Moreover,  these sequences are uniformly bounded in $L^{\infty}$, as  $\|(\zeta_{\gk},\omega_{\gk})\|_\infty \leq {M}$, $\|\theta_{\gk}\|_\infty \leq K$ and $\|\vartheta_{\gk}\|_\infty\leq 2M_{\psi}$.\vspace{0.2cm}\\
{\bf Step 1.} Construction of $\xxi$, the admissible equation.\vspace{0.1cm} \\
From Theorem \ref{lasthopefullybis}, we have that the triplet $(\x,\y,\bar{\xi})$ satisfies \eqref{admissible-P} for all $t\in [0,1]$. Hence, for all $t\in[0,1]$ we have
$$\dot{\x}(t)=f_{\Vphi}(\x(t), \u(t))-\xxi(t) \nabla\psi(\x(t))=f(\x(t), \u(t))-\nabla\Vphi(\x(t))-\xxi(t)\nabla\psi(\x(t)).$$
Since $\nabla \Vphi(x)= \partial_{\ell}\varphi(x)=\nabla_{\bp\ell}\sp\sp\varphi(x)$ for all $x\in C$, we obtain that \begin{equation*} \dot{\x}(t)= f(\x(t),\u(t))- \nabla_{\bp\ell}\sp\sp\varphi(\x(t))-\xxi(t) \nabla\psi(\x(t)),\;\; \forall\sp t\in [0, 1]. \end{equation*}
On the other hand, since $\bar{x}$ takes values in $C$, we have $\psi(\x(t))\le 0$, $\forall t\in [0,1].$\vspace{0.2cm} \\
{\bf Step 2.} Construction of $\p$, $\zzeta$, $\ttheta$, $\vvartheta$, $\oomega$, $\nnu$, and the adjoint equation. \vspace{0.1cm} \\
For the construction of $\p$, $\ttheta$, $\vvartheta$ and $\nnu$, see Steps 2-4 in the proof of \cite[Theorem 5.1]{verachadi}. Note that the uniform boundedness of $p_{\gk}(1)$ established and used in Step 2 of the proof of \cite[Theorem 5.1]{verachadi}, is easily deduced here from the nontriviality condition of Proposition \ref{mpw12aproxw12}. We also note that,  similarly to Step 2 of the proof of \cite[Theorem 5.1]{verachadi},  $p_{\gk}$  has a uniformly bounded variation, and hence,  Helly first theorem implies that $p_{\gk}$ admits a pointwise convergent subsequence whose limit $\p$ is also of bounded variation and satisfies, for some $M_1>0$, the following \begin{equation}\label{ntc2} \|\p\|_\infty\leq M_1\|\p(1)\|. \end{equation} Using Helly second theorem we obtain that for all $h\in C([0,1];\R^n)$, 
\begin{equation}\label{mp16}\lim_{k\to\infty}  \int_{[0,1]} \<h(t), \dot{p}_{\gk}(t)\> \,dt =\int_{[0,1]}\<h(t),d\p(t)\>.\end{equation}
Identically  to Steps 2-4 in the proof of \cite[Theorem 5.1]{verachadi}, we also have
\begin{equation}\label{mp15} \int_0^1 \left\<h(t),\theta_{\gk}(t) \,p_{\gk}(t)\right\>\sp dt \f \int_0^1\left\<h(t),\ttheta(t)\, \p(t)\right\>\sp dt,\end{equation} 
\begin{equation} \label{mp17} \int_0^1 \xi_{\gk}(t)\; \<h(t),\vartheta_{\gk}(t)\,p_{\gk}(t)\>\,\sp dt\f\int_0^1\xxi(t)\;\<h(t),\vvartheta(t)\, \p(t)\>\,\sp dt,
\end{equation}
\begin{eqnarray} \nonumber&&\lim_{k\f\infty} \int_0^1\<h(t),\nabla\psi(x_{\gk}(t))\>\,\gk \xi_{\gk}(t)\,\<\nabla\psi(x_{\gk}(t)),p_{\gk}(t)\>\,\sp dt\\[3pt]&=&\int_{[0,1]} \<h(t),\nabla\psi(\x(t))\> \,\sp d\nnu(t).\label{mp18}\\[-4pt]\nonumber\end{eqnarray}

We proceed to construct the two functions $\zzeta$ and $\oomega$. Note that the construction of $\zeta$ done in Step 2 of the proof of \cite[Theorem 5.1]{verachadi} cannot be used here, since  the {\it closed graph} hypothesis on the multifunction $(x,u)\mapsto \partial^{\sp x} f(x,u)$ is required there, but it is not assumed here. As the sequence $(\zeta_{\gk},\omega_{\gk})_k$ is uniformly bounded in $L^{\infty}$, it has a subsequence, we do not relabel, that converges weakly in $L^1$  to some $(\zzeta,\oomega)$.  Using that the multifunction $(x,u)\mapsto \partial^{\sp (x,u)}_\ell f(x,u)$ has closed graph with nonempty, compact and convex values, \cite[Theorem 6.39]{clarkebook}  implies that, for $t\in [0,1]$ a.e., 
\begin{equation*}(\zzeta(t),\oomega(t))\in \partial^{\sp (x,u)}_\ell f(\x(t),\u(t)).	
\end{equation*}
Since $(p_{\gk})_{k}$ is uniformly bounded in $L^\infty$ and converges pointwise to $\p$, we conclude that \begin{equation} \label{w12mp2}\zeta_{\gk}\tran p_{\gk}\xrightarrow[{L^1}]{\textnormal{weakly}}\zzeta\tran \p\;\;\hbox{and}\;\;\omega_{\gk}\tran p_{\gk}\xrightarrow[{L^1}]{\textnormal{weakly}}\oomega\tran \p.\end{equation} 
Hence, for all $h\in C([0,1];\R^n)$,
\begin{equation}\label{mp19} \int_0^1 \left\<h(t),\zeta_{\gk}(t) p_{\gk}(t)\right\>\sp dt \f \int_0^1\left\<h(t),\zzeta(t)  \p(t)\right\>\sp dt.\end{equation} 
Thus, from \eqref{mpim1} and \eqref{mp16}-\eqref{mp19}, we conclude that the adjoint equation of Theorem \ref{mpw12bisw12} holds, and it coincides with the adjoint equation of \cite[Theorem 5.1]{verachadi}.\vspace{0.2cm} \\
{\bf Step 3.} The complementary slackness conditions.\vspace{0.1cm} \\
The part $(a)$ follows from the equation \eqref{boundxi}. The part $(b)$ follows from the uniform boundedness of $\left\|\gk \xi_{\gk}(\cdot)\<\nabla\psi(x_{\gk}(\cdot)),p_{\gk}(\cdot)\>\right\|_1$ established in \cite[Equation (97)]{verachadi}. More details can be found  in Step 6 of \cite[Proof of Theorem 6.1]{verachadiimp}.\vspace{0.2cm} \\
{\bf Step 4.} Construction of $\l$ and the transversality equation.\vspace{0.1cm} \\
Form the transversality condition of Proposition \ref{mpw12aproxw12}, there exist $\upsilon_{\gk}\in N^L_{C_0(k)}(x_{\gk}(0))$, $\chi_{\gk}\in N_{C_1(k)}^L(x_{\gk}(1))$  and $(a_{\gk},b_{\gk})\in \partial^L g(x_{\gk}(0),x_{\gk}(1))$ such that \begin{equation}\label{willdiesoon} p_{\gk}(0)=\l_{\gk}a_{\gk} +\l_{\gk}(x_{\gk}(0)\bbp-\bbp\x(0)) + \upsilon_{\gk},\;\;\;-p_{\gk}(1)=\l_{\gk}b_{\gk}+\chi_{\gk},\end{equation}
and the following properties hold:
\begin{itemize}[leftmargin=*]
\item $\|(a_{\gk}, b_{\gk})\|\leq L_g$, where $L_g$ is the Lipschitz constant of $g$ over $\tilde{C}_0(\delta)\times \tilde{C_1}(\delta)$, and $\|\l_{\gk}\|\leq 1$ for all $k$. The latter inequality gives  a subsequence, we do not relabel, such that $\l_{\gk}\f\l\in [0,1]$.
\item Due to Theorem \ref{lasthopefullybis}, we have, for  $k$ large enough, 
 $$(x_{\gk}(0), x_{\gk}(1))\in \inte (\tilde{C}_0(\delta) \times \tilde{C}_1(\delta)),\,\text{and}\,(x_{\gk}(0),  x_{\gk}(1))\f (\x(0), \x(1)).$$ 
\item We have $p_{\gk}(0)\f \p(0)$ and  $p_{\gk}(1)\f \p(1)$.
\item Owing to \eqref{normalC1}, in which $d:= x_{\gk}(1)\in \big[C_1(k)\cap (\inte C)\cap B_{\delta_{\bbbp o}}(\x(1))\big]$ for $k$ sufficiently large, we have $\chi_{\gk}\in N_{C_1(k)}^L(x_{\gk}(1))= N_{C_1}^L\left(x_{\gk}(1)+\x(1)-\x_{\gk}(1)\right)$, for $k$ large,  where, we recall that  $\x_{\gk}(1)\f\x(1)$.
\item Owing to \eqref{normalC0}, in which  $c:=x_{\gk}(0)\in C_0(k)\cap B_{\delta_{\bbbp o}}(\x(0))$ for $k$ large enough, it follows that:
\begin{enumerate}[$(i)$,leftmargin=0.8cm] 
\item If $\x(0)\in\inte C$, then for $k$ sufficiently large $$v_{\gk}\in N_{C_0(k)}^L(x_{\gk}(0))= N_{C_0}^L(x_{\gk}(0)).$$
\item If $\x(0)\in\bdry C$, using that $x_{\gk}(0)\f \x(0)$ and $\rho_{k}\f 0$, then for $k$ sufficiently large, $\left(x_{\gk}(0) +\rho_{k}\frac{\nabla\psi(\x(0))}{\|\nabla\psi(\x(0))\|}\right)\in{B}_{\delta_{\bbbp o}}(\x(0))$, and hence,
 $$\textstyle v_{\gk}\in N_{C_0(k)}^L(x_{\gk}(0))= N_{C_0}^L \left(x_{\gk}(0) +\rho_{k}\frac{\nabla\psi(\x(0))}{\|\nabla\psi(\x(0))\|}\right)\;\,\hbox{for}\; k\;\hbox{large}.$$
  \end{enumerate}
 \end{itemize}
 Therefore, along a subsequence of $(\gk)_k$, we do not relabel, we have
 $$	\l_{\gk}(a_{\gk}, b_{\gk})\f \l(a,b)\in \l\partial^L_{\ell} g(\x(0),\x(1))\;\;\hbox{and}\;\;\l_{\gk}(x_{\gk}(0)\bbp-\bbp\x(0))\f 0.$$
 Thus, taking the limit as $k\to\infty$ in \eqref{willdiesoon}, and  using   $(p_{\gk}(0),p_{\gk}(1))\f (\p(0),\p(1))$, we obtain that $(v_{\gk},\chi_{\gk})$ must converge to some $(v,\chi)$, as all the other terms in \eqref{willdiesoon} converge.  The last two bullets, stated above, yield that $v\in N^L_{C_0}(\x(0))$\; and $\chi\in  N_{C_1}^L(\x(1))$. Consequently, the limit of \eqref{willdiesoon} is equivalent to  $$ (\p(0),-\p(1))\in  \l\partial_\ell^L g(\x(0),\x(1))+ \big[N_{C_0}^L(\x(0))\times\;N^L_{C_1}(\x(1))\big];$$ 
This terminates the proof of the transversality equation.\vspace{0.2cm}\\
{\bf Step 5.} The weak maximization condition.\vspace{0.1cm} \\
By \eqref{w12mp1},  \eqref{w12mp2}(b), and the transversality equations of Proposition \ref{mpw12aproxw12},  we have that
\begin{equation} \label{lastproof6} \dot{q}_{\gk}=-	(\omega_{\gk})\tran p_{\gk}\xrightarrow[{L^1}]{\textnormal{weakly}}-(\oomega)\tran \p\;\;\hbox{and}\;\;q_{\gk}(0)=\l_{\gk}(u_{\gk}(0)-\u(0)).\end{equation}
The uniform boundedness in $L^{\infty}$ of the sequences $(p_{\gk})_{k}$ and $(\omega_{\gk})_{k}$ give that $(\dot{q}_{\gk})_{k}$ is uniformly bounded in $L^{\infty}$, asserting the equicontinuity of $(q_{\gk})_{k}$. Moreover, the nontriviality condition of Proposition \ref{mpw12aproxw12} gives the uniform boundedness of the sequence $(q_{\gk})_{k}$. Hence, by Arzel\`a-Ascoli theorem, the sequence $(q_{\gk})_{k}$ admits a subsequence, we do not relabel, that converges uniformly to an absolutely continuous function $q$ satisfying $q(0)=0$ (by (\ref{lastproof6})(b), where  $\l_{\gk}\f\l$ and $u_{\gk}(0)\f\u(0)$ as $k\f \infty$). Moreover, up to a subsequence, we also obtain that \begin{equation} \label{lastproof7} \dot{q}_{\gk} \xrightarrow[{L^1}]{\textnormal{weakly}}\dot{q}.\end{equation}   
Hence, (\ref{lastproof6})(a) and the uniqueness of the $L^1$-weak limit yield that \begin{equation} \label{lastproof13} \dot{q}(t)=-(\oomega(t))\tran \p(t),\;\; \sp t\in [0,1]\; \textnormal{a.e.}\end{equation} 

We proceed to study the convergence of the sequence of $NBV$-functions, $(\Omega_{\gk})_k$, obtained in Proposition \ref{mpw12aproxw12}.  The maximization condition $(iv)$, therein, implies  that, for $t\in[0,1]$ a.e.,  
\begin{equation} \label{lastproof8} \Omega_{\gk}(t)=-q_{\gk}(t)+\underbrace{\l_{\gk}(\dot{u}_{\gk}(t)-\dot{\bar{u}}(t))}_{\ell_{\gk}(t)}.\end{equation}
Without loss of generality, we can assume that  \eqref{lastproof8} is satisfied for all $t\in[0,1]$.  In fact, if $\l_{\gk}=0$, using the transversality conditions of Proposition \ref{mpw12aproxw12} and that $\Omega_{\gk}\in NBV[0,1]$,   we get that $\Omega_{\gk}(0)=-q_{\gk}(0)=0$ and  $\Omega_{\gk}(1)=-q_{\gk}(1)$, and hence,  by the right continuity of $\Omega_{\gk}$ and the continuity of $q_{\gk}$,  \eqref{lastproof8}  is equivalent to  $\Omega_{\gk} \equiv - q_{\gk}$. If, however,  $\l_{\gk} >0$, then by modifying the values of $(\dot{u}_{\gk}-\dot{\bar{u}})$ on  the set of Lebesgue measure zero, we  have \eqref{lastproof8} satisfed for all $t\in[0,1]$, and hence, $\ell_{\gk}\in BV[0,1]$ is right continuous on $(0,1)$, and satisfies  $\ell_{\gk}(0)= q_{\gk}(0)$ and $\ell_{\gk}(1)=0$. Furthermore, 
since  $\l_{\gk}\f \l$ and $\dot{u}_{\gk}$ strongly converges in $L^2$ to $\dot{\u}$, the sequence $(\ell_{\gk})_k$ strongly converges in $L^2$ to $\ell=0$.
 
We claim that $(\Omega_{\gk})_k$, considered  as a sequence of continuous linear functionals on $C([0,1];\R^m)$,  admits  a subsequence, we do not relabel, that converges weakly* to $-q$. Since $\Omega_{\gk}$ satisfies \eqref{lastproof8}, where the sequence of absolutely continuous functions $(q_{\gk})_k$ converges uniformly to $q\in AC([0,1];\R^m)$ and, by (\ref{lastproof7}), $(\dot{q}_{\gk})_k$ converges weakly in $L^1$ to $\dot{q}$, then it is equivalent  to show that  the $BV$-sequence $(\ell_{\gk})_k$ converges in $C^*([0,1];\R^m)$ to $0$.
 The uniform boundedness of the sequence   $(\ell_{\gk})_k$ shall follow once we show  the uniform boundedness of  $(\Omega_{\gk})_k$.  For this latter,  the nontriviality condition of Proposition \ref{mpw12aproxw12}, implies that the sequence $(\mu^{\bbbp o}_{\gk})_k$ is uniformly bounded, and hence, it has a subsequence, we do not relabel, that converges weakly* to a $\mu^{\bbbp o}\in C^{\oplus}([0,1];\R^{m})$. Moreover, by condition $(v)$ of Proposition \ref{mpw12aproxw12} and Remark \ref{CQremark}$(i)$,  $\|\b_{\gk}(t)\|\le 1$, except on a set of $\mu^{\bbbp o}_{\gk}$-measure zero. Thus, using that  \begin{equation}\label{omegaformula} \Omega_{\gk}(t)= \int_{[0,t]} \beta_{\gk}(s)\mu^{\bbbp o}_{\gk}(ds),\;\;\forall t\in(0,1],\;\;\hbox{and}\;\;\Omega_{\gk}(0)=0,\end{equation} we obtain that the sequence $(\Omega_{\gk})_k$ is uniformly bounded, and so is the sequence $(\ell_{\gk})_k$.  Hence, to get that the bounded $BV$-sequence $(\ell_{\gk})_k$ converges weakly* to $0$, by the Banach-Steinhaus theorem in \cite[p.482]{kadets}, it is sufficient to show that 
$$\lim_{k\f \infty}\int_0^1 \<h(t),d{\ell}_{\gk}(t)\>=0,\;\;\forall h\in C^1([0,1];\R^m).$$
Fix $h\in C^1([0,1];\R^m)$. Using an integration by parts and that $\Omega_{\gk}\in NBV,$ we get
\begin{eqnarray*}\int_0^1 \<h(t),d{\ell}_{\gk}(t)\> &=&\<h(1),\ell_{\gk}(1)\>-\<h(0),\ell_{\gk}(0)\>-\int_0^1 \left\<\dot{h}(t),\ell_{\gk}(t)\right\>\sp dt\\ &=&\left[-\<h(0),q_{\gk}(0)\> -\int_0^1 \left\<\dot{h}(t),\ell_{\gk}(t)\right\>\sp dt\right]\;  \xrightarrow[{k\f \infty}]{}0,
	\end{eqnarray*}
since $(\ell_{\gk})_k$  strongly converges  in $L^2$ to $0$, and $q_{\gk}(0)\f q(0)=0$. This terminates the proof of the claim, that is, 
\begin{equation} \label{lastproof9}\Omega_{\gk}\xrightarrow[{C^*([0,1];\R^m)}]{\textnormal{weakly*}} -q.\end{equation}
By \cite[p. 484, \#8]{kadets},  we also have that $ \Omega_{\gk}(t)\f -q(t),\;\forall\sp t\in[0,1].$

Now define the signed measure $\mu_{\gk}(dt):=\beta_{\gk}(t)\mu_{\gk}^{\bbbp o}(dt)$. From \eqref{omegaformula} we have \begin{equation} \label{omegatomu}\Omega_{\gk}(t)= \mu_{\gk}[0,t], \;\;\forall t\in(0,1].\end{equation}
Using  that $\mu_{\gk}^{\bbbp o}\xrightarrow[{k\f \infty}]{\textnormal{weakly*}}\mu^{\bbbp o}$, and that  Proposition \ref{mpw12aproxw12}$(v)$ holds true,  then, by  applying \cite[Proposition 9.2.1]{vinter} to the following data: \begin{itemize}[leftmargin=*]
 \item $A_{\gk}(t):=\partial_u^{>}d(u_{\gk}(t), U(t)\cap \bar{B}_{\delta}(\u(t))) $ for all $t\in [0,1]$,
 \item $A(t):=\partial_u^{>}d(\u(t), U(t)\cap \bar{B}_{\delta}(\u(t)))$ for all $t\in [0,1]$,
 \item $\gamma_{\gk}:=\beta_{\gk}$, $\mu_{\gk}:=\mu_{\gk}^{\bbbp o}$ and $\mu_0:=\mu^{\bbbp o}$, \end{itemize}
we obtain a Borel measurable function $\beta\colon [0,1]\f\R^m$ and $\mu\in C^*([0,1];\R^m)$ such that
$$\mu_{\gk}\xrightarrow[{k\f \infty}]{\textnormal{weakly*}}\mu,\;\;\mu(dt)=\beta(t)\mu^{\bbbp o}(dt)\;\hbox{and}\;\beta(t)\in \partial_u^{>}d(\u(t), U(t)\cap \bar{B}_{\delta}(\u(t)))\;\;\mu^{\bbbp o}\;\textnormal{a.e.}$$
Since  $u_{\gk}$ converges uniformly to $\bar{u}$, and $\supp\{\mu_{\gk}^{\bbbp o}\}$ satisfies  Proposition \ref{mpw12aproxw12}$(v)$, we deduce that
\begin{equation}\label{scriptA}\supp\{\mu^{\bbbp o}\} \subset  \mathcal{A}:=\left\{t\in [0,1]: (t,\bar{u}(t))\in \bdry\Gr\left[U(t)\cap \bar{B}_{\delta}(\u(t))\right]\right\}.\end{equation} 
Adjust $\b(\cdot)$ on the set of $\mu^{\bbbp o}$-measure zero to arrange
$$t\in \mathcal{A} \implies 
\beta(t)\in \partial_u^{>}d(\u(t), U(t)\cap \bar{B}_{\delta}(\u(t))),$$
and hence,  using \cite[Formula (9.17)]{vinter}, we have 
\begin{equation}\label{wmc2} \beta(t)\in \left[\conv\bar{N}^L_{U(t)\cap \bar{B}_{\delta}(\u(t))}(\u(t))\cap \left(\bar{B}\setminus\{0\}\right)\right], \;\;\forall\sp t\in \mathcal{A}.\end{equation}
Thus, by \eqref{lastproof9} and \eqref{omegatomu}, we obtain that  $-dq(t)=\mu(dt)=\beta(t) \mu^{\bbbp o}(dt)$.  Using \eqref{lastproof13}, we arrive to
\begin{equation}\label{wmc1}-dq(t)= (\oomega(t))\tran \p(t)\sp dt=\beta(t)\mu^{\bbbp o}(dt).\end{equation} 
Next, we decompose $\mu^{\bbbp o}(dt)=m(t)dt+\mu_s(dt)$ for some {\it nonnegative} $L^1$-function $m(\cdot)$ and some nonnegative Borel measure $\mu_s$ totally singular with respect to Lebesgue measure.  Clearly $m(t)=0$, for all 
$t\in \mathcal{A}^c $, and hence, \eqref{wmc2} implies that 
$$ \beta(t) m(t) \in \conv\bar{N}^L_{U(t)\cap \bar{B}_{\delta}(\u(t))}(\u(t)), \;\; \forall \sp t\in[0,1].$$
Using \eqref{wmc1} we get that $(\oomega(t))\tran \p(t)\sp dt=\beta(t)m(t)dt+\beta(t)\mu_s(dt).$ This gives that  $(\oomega(t))\tran \p(t)= \beta(t)m(t)$, for $t\in [0,1]$ a.e. Therefore, \begin{equation}\label{wmc3} \oomega(t)\tran \p(t)\in \conv\bar{N}^L_{U(t)\cap \bar{B}_{\delta}(\u(t))}(\u(t)),\;\; \forall\sp t\in [0,1]\; \textnormal{a.e.}\end{equation}

We proceed to prove that the ``\sp In addition" part of the weak maximization condition. We assume the existence of $\e_{\bbbp o}>0$ and $r>0$ such that $U(t)\cap \bar{B}_{\e_{\bbbp o}}(\u(t))$ is $r$-prox-regular for all $t\in [0,1]$. From \eqref{wmc3} and using Lemma \ref{newlemmaprox}, we obtain that for all $t\in[0,1]$ a.e. $$\textstyle  \left\<\oomega(t)\tran \p(t),u-\u(t)\right\>\leq \frac{\|\oomega(t)\tran \p(t)\|}{\min\{\e_{\bbbp o},2r\}}\|u-\u(t)\|^2,\;\;\hbox{for all}\;u\in U(t).$$
Therefore, for \textnormal{a.e.} $t\in [0,1]$, $$\textstyle \left\<\omega(t)\tran p(t),u\right\>- \frac{\|\oomega(t)\tran \p(t)\|}{\min\{\e_{\bbbp o},2r\}}\|u-\u(t)\|^2  \leq \left\<\oomega(t)\tran \p(t),\u(t)\right\> ,\;\;\hbox{for all}\;u\in U(t).$$
This terminates the proof of the weak maximization condition.\vspace{0.2cm} \\
{\bf Step 6.} The nontriviality condition.\vspace{0.1cm}\\
It is sufficient to prove its equivalent condition: $\|\p(1)\| + \l\not=0$. Taking the  limit   as $k\f\infty$ in the nontriviality condition of Proposition \ref{mpw12aproxw12}, and using the convergence of $\p_{\gk}(1)$ to $\p(1)$, the uniform convergence of $q_{\gk}$ to $q$, the weak* convergence of $\mu^{\bbbp o}_{\gk}$ to $\mu^{\bbbp o}$, and the convergence of $\l_{\gk}$ to $\l$, we get that \begin{equation}  \label{lastproof12} 1 = \|\p(1)\|+ \|q\|_\infty + \|\mu^{\bbbp o}\|\TV + \l.\end{equation}
We argue by contradiction. If  $\p(1)=0$ and $\l=0$, by \eqref{ntc2} we obtain that $\p=0$. Hence (\ref{lastproof13}) yields that $\dot{q}(t)=0$ for a.e. $t\in[0,1]$. This gives that \begin{equation} \label{ntc3} \beta(t) \mu^{\bbbp o}(dt)=-dq(t)=0\;\;\hbox{and}\;\;q(t)=q(0)+\int_0^t \dot{q}(\tau)\sp d\tau=0,\;\;\forall\sp t\in[0,1].\end{equation}	
Since, by  \eqref{scriptA} and \eqref{wmc1},  $\supp\{\mu^{\bbbp o}\} \subset  \mathcal{A}$ and $\beta(t)\not=0$ for all $t\in\mathcal{A}$,  the first equation of \eqref{ntc3} yields that $\mu^{\bbbp o}=0$. Therefore, $\|\p(1)\| + \|q\|_\infty + \|\mu^{\bbbp o}\|\TV + \l=0$ which contradicts (\ref{lastproof12}). This terminates the proof of the nontriviality condition, and then, the proof of the conditions $(i)$-$(vi)$ of Theorem \ref{mpw12bisw12} is completed.  \end{proof}

%-------------
\section{Appendix} \label{auxresults} In this section, we present the proofs of Theorems \ref{prop3new} and \ref{lasthopefullybis}, and we establish auxilliary results that are used in different places of the paper. We begin by the proof of Theorem \ref{prop3new}.

\begin{proof}\bp{\it of Theorem} \ref{prop3new}. $(i)$: Having a uniform bounded derivative in $L^2([0,1];\R^m)$, the $W^{1,2}$-sequence $u_{\gk}$  is equicontinuous. Since, by (A4.2),   the compact sets $U(t)$ are uniformly bounded,  then $u_{\gk}$ is uniformly bounded in $C([0,1];\R^m)$, and hence,  Arzel\`a-Ascoli theorem asserts that  $u_{\gk}$ admits a subsequence, we do not relabel, that converges uniformly to an absolutely continuous function $u$ with $u(t)\in U(t)$ for all $t\in [0,1]$. As   $\dot{u}_{\gk}$ is uniformly bounded in $L^2([0,1];\R^m)$,  then, up to a subsequence,  it is  weakly convergent in $L^2$.  The boundedness of $(u_{\gk}(0))_k$ then yields that the $L^2$-weak limit of  $\dot{u}_{\gk}$ is $\dot{u}$, and whence, $u\in \mathscr{W}$. The fact that $x$  is the unique solution to  $(D)$ corresponding to  $(x_0,u)$, and the proceeding statements of this part, follow immediately from Theorem \ref{prop2}$(ii)$. \vspace{0.2cm}\\
 $(ii)$:  Now, assume that $c_{\gk}\in C(k)$ for $k\ge k_{o}$, where $k_{o}$ is the rank in Theorem \ref{propooufnew}.  
 
 Let us first show that $(\xi_{\gk})_k$ has uniform bounded variations. Since $\psi$ and $\nabla\psi$ are  Lipschitz on $C$ and $x_{\gk}$ is Lipschitz for $k\ge k_o$, we deduce that, for $k\ge k_o$, the function $\xi_{\gk}(\cdot)\nabla\psi(x_{\gk}(\cdot))$,  where $\xi_{\gk}  $ is defined in (\ref{defxi}), is Lipschitz continuous on $[0,1]$. Similarly,  the  Lipschitz property on $C\times (\mathbb{U} +\tilde{\rho}\bar{B})$  of $f(\cdot,\cdot)$  (and then of $f_{\Vphi}(\cdot,\cdot)$) and the fact that  $(x_{\gk}, u_{\gk})$ is in $W^{1,\infty}\times W^{1,2}$, yield that  $f_{\Vphi}(x_{\gk}(\cdot),u_{\gk}(\cdot))$ is in $W^{1,2}([0,1];\R^n)$. Hence,  
 \begin{equation*}\label{zetak}
 \zeta_{\gk}(t):=\frac{d}{dt}{f_{\Vphi}}(x_{\gk}(t),u_{\gk}(t)),\end{equation*} 
 exists for almost all $t\in[0,1]$. By writing  
 $f_{\Vphi}= (f^1_{\Vphi},\cdots, f^n_{\Vphi})^{\tran}$, and using that $x_{\gk}(t)\in \inte C$ (for all $t\in [0,1])$, and $u_{\gk}(t) \in U(t)\subset \mathbb{U}$ (for $t\in[0,1]$\, a.e.), it follows from the proof of \cite[Theorem 2.1]{zeidan}, that 
 \begin{equation*}
 \zeta^i_{\gk}(t)\in \<\partial f^i_{\Vphi}(x_{\gk}(t),u_{\gk}(t)), (\dot{x}_{\gk}(t), \dot{u}_{\gk}(t))\>, \;\; t\in[0,1]\; \hbox{a.e.},  \; \forall \sp i=1,\cdots, n.
  \end{equation*}
Since $(\|\dot{u}_{\gk}\|_2)_k$ is assumed to be bounded, and,  by Theorem  \ref{propooufnew},  $(\|\dot{x}_{\gk}\|_{\infty})_k$ is  bounded, then the sequence $(\|\zeta_{\gk}\|_2)_k$ is bounded by some $M_{\zeta} >0$ that depends  on $\bar{M}$, $\bar{M}_\psi$, $\eta$,  and the bound of   $(\|\dot{u}_{\gk}\|_2)_k$. 

As  $f_{\Vphi}(x_{\gk}(\cdot),u_{\gk}(\cdot)) \in W^{1,2}([0,1];\R^n)$, the right hand side of $(D_{\gk})$ yields that  $\dot{x}_{\gk}$ is in $W^{1,2}([0,1];\R^n)$, and so is the function $\lvert\<\nabla\psi(x_{\gk}(\cdot)),\dot{x}_{\gk}(\cdot)\>\rvert$. This also implies that $\xi_{\gk}\in W^{2,2}([0,1];\R^{+})$, due to 
 $$\dot{\xi}_{\gk}(t)=\gk^2 e^{\gk \psi(x_{\gk}(t))}\<\nabla\psi(x_{\gk}(t)),\dot{x}_{\gk}(t)\>.$$ 
Next, calculating $\ddot{x}_{\gk}$ through ($D_{\gk}$) in terms of $\zeta_{\gk}$ and $\dot{x}_{\gk}$, and using the fact that for $h\in AC([0,1];\R)$ we have \begin{equation*}\label{discov} \frac{d}{dt} \lvert h(t)\rvert=\left(\frac{d}{dt}h(t)\right)\sign(h(t))\;\;\hbox{a.e.}\;t\in(0,1),\footnote{The function $\sign\colon\R\f\R$ is defined by: $ \sign(x)=\frac{x}{\lvert x \rvert}$ for $x\not=0$, and $0$ for $x=0$.}\end{equation*}
 it follows  that  there exist measurable functions  $\vartheta^1_{\gk}$ and $\vartheta^2_{\gk}$ whose values at $t$ are in  $\partial^2\psi(x_{\gk}(t))$, for almost all $t\in  [0,1]$, such that, for  $t\in [0,1]$ a.e., we have
\begin{eqnarray*}&&\frac{d}{dt} \lvert\<\nabla\psi(x_{\gk}(t)),\dot{x}_{\gk}(t)\>\rvert\\[-0.3cm]&=&\left[\left\<\vartheta^1_{\gk}(t)\dot{x}_{\gk}(t),\dot{x}_{\gk}(t)\right\>+\left\<\nabla\psi(x_{\gk}(t)),\ddot{x}_{\gk}(t)\right\>\right]\overbrace{\sign (\<\nabla\psi(x_{\gk}(t)),\dot{x}_{\gk}(t)\>)}^{\a(t)}\\[3pt] &=&\left[\left\<\vartheta^1_{\gk}(t)\dot{x}_{\gk}(t),\dot{x}_{\gk}(t)\right\>+ \<\nabla\psi(x_{\gk}(t)),\zeta_{\gk}(t) - \xi_{\gk}(t)\vartheta^2_{\gk}(t)\dot{x}_{\gk}(t)\>\right]\a(t)\\[3pt]&& - \<\underbrace{\gk \xi_{\gk}(t)\<\nabla \psi(x_{\gk}(t)),\dot{x}_{\gk}(t)\>}_{\dot{\xi}_{\gk}(t)} \nabla\psi(x_{\gk}(t)), \nabla\psi(x_{\gk}(t))\>\a(t)\\[3pt] &=&\left[\left\<\vartheta^1_{\gk}(t)\dot{x}_{\gk}(t),\dot{x}_{\gk}(t)\right\>+ \left\<\nabla\psi(x_{\gk}(t)),\zeta_{\gk}(t)- \xi_{\gk}(t)\vartheta^2_{\gk}(t)\dot{x}_{\gk}(t)\right\>\right]\a(t)\\[3pt]&& - \gk \xi_{\gk}(t)\, 
\lvert\<\nabla \psi(x_{\gk}(t)),\dot{x}_{\gk}(t)\>\rvert\;\|\nabla\psi(x_{\gk}(t))\|^2.
\end{eqnarray*}
Integrating both sides on $[0,1]$ and using the boundedness of $(\|\dot{x}_{\gk}\|_{\infty})_k$ and $(\|\xi_{\gk}\|_{\infty})_k$ (by Theorem \ref{propooufnew}), and assumption (A2.1), we get the existence of a constant $\tilde{M}_1$ depending on $\bar{M}$, $M_\psi$, $\bar{M}_\psi$, $\eta$, and $M_{\zeta}$ such that \begin{eqnarray*} \int_0^1 \lvert\dot{\xi}_{\gk}(t)\rvert \|\nabla\psi(x_{\gk}(t))\|^2\sp dt\leq \tilde{M}_1.\\[-0.2cm] \end{eqnarray*}
Using (\ref{velo0}) and assumption (A2.2), it follows that 
\begin{eqnarray*}&& \int_0^1 \lvert\dot{\xi}_{\gk}(t)\rvert\sp dt = \int_0^1 \gk^2 e^{\gk \psi(x_{\gk}(t))}\lvert\<\nabla\psi(x_{\gk}(t)),\dot{x}_{\gk}(t)\>\rvert\sp dt \\[3pt]&=&  \int_{\{t : \|\nabla\psi(x_{\gk}(t))\|\leq \eta\}}\gk^2 e^{\gk\psi(x_{\gk}(t))} \lvert\<\nabla\psi(x_{\gk}(t)),\dot{x}_{\gk}(t)\>\rvert\sp dt\\[3pt]&+&\int_{\{t : \|\nabla\psi(x_{\gk}(t))\|> \eta\}}\gk^2 e^{\gk\psi(x_{\gk}(t))}\lvert\<\nabla\psi(x_{\gk}(t)),\dot{x}_{\gk}(t)\>\rvert\frac{\|\nabla\psi(x_{\gk}(t))\|^2}{\|\nabla\psi(x_{\gk}(t))\|^2}\sp dt \\[4pt] &\leq &  \eta\left(\bar{M}+\frac{2\bar{M}\bar{M}_\psi}{\eta}\right)\gk^2 e^{-\gk\e} + \frac{\tilde{M}_1}{\eta^2}\leq\eta\left(\bar{M}+\frac{2\bar{M}\bar{M}_{\psi}}{\eta}\right)+ \frac{\tilde{M}_1}{\eta^2}=:\tilde{M}_2, \end{eqnarray*}
for $k$ sufficiently large, where $\tilde{M}_2$ depends  on the given constants, $\bar{M}$, $\bar{M}_\psi$, $M_{\psi}$, $\eta$, and on the bound of $(\|\dot{u}_{\gk}\|_2)_k$. Therefore, the sequence ${\xi}_{\gk}$ satisfies, for $k$ sufficiently large,
$V_0^1(\xi_{\gk})\leq \tilde{M}_2.$

On the other hand, by Theorem \ref{propooufnew},     $\|\xi_{\gk}\|_{\infty}\le  \frac{2\bar{M}}{\eta}$ for all $k\ge k_o$.  Hence,   by Helly first theorem, $\xi_{\gk}$ admits a pointwise convergent subsequence, we do not relabel, whose limit is some function $\tilde{\xi}\in BV([0,1];\R^{+})$ with $\|\tilde{\xi}\|_{\infty}\le  \frac{2\bar{M}}{\eta}$ and $V_0^1(\tilde{\xi})\leq \tilde{M}_2$.  Being pointwise convergent to $\tilde{\xi}$ and uniformly bounded in $L^{\infty}$, $\xi_{\gk}$ {\it strongly} converges in $L^2$ to $\tilde{\xi}$.
However,  by part$(i)$ of this theorem, $\xi_{\gk}$ converges weakly in $L^2$ to $\xi$, hence,  $\tilde{\xi}=\xi$. Thus,  $\xi_{\gk}$ converges pointwise and strongly in $L^2$ to $\xi$, and $\xi \in BV([0,1];\R^{+})$  with  \begin{equation} \label{existencenew} V_0^1(\xi)\leq \tilde{M}_2.
\end{equation}
As $f$ is $M$-Lipschitz on  $C\times (\mathbb{U}+\tilde{\rho}\bar{B})$, $u\in \mathscr{W}$,  $\nabla\psi$ is Lipschitz, and  $\xi\in BV$, then equation \eqref{admissible-P}, which is satisfied by $(x,u,\xi)$,  now holds for all $t\in[0,1]$. This yields that \eqref{boundxi} is also valid for all $t\in[0,1]$, and  that  $\dot{x}\in BV([0,1];\R^n)$.

It remains to show that $\dot{x}_{\gk}$  has uniform  bounded variations and converges pointwise and strongly in $L^2$ to $\dot{x}\in BV([0,1];\R^n)$. 
 Since  $\xi_{\gk}(\cdot)\nabla\psi(x_{\gk}(\cdot))$ is Lipschitz, $u_{\gk}\in \mathscr{W}$, and $f$ is $M$-Lipschitz on $C\times (\mathbb{U}+\tilde{\rho}\bar{B})$, then $(D_{\gk})$ holds for all $t\in [0,1]$, that is, 
$$\dot{x}_{\gk}(t)= f_{\Vphi}(x_{\gk}(t),u_{\gk}(t))-\xi_{\gk}(t)\nabla\psi(x_{\gk}(t)), \;\;\;\forall\,t\in [0,1].$$
Hence, using  part$(i)$ of this theorem,  the continuity of $f_{\Vphi}(\cdot, \cdot)$,   that  the sequence $(x_{\gk},u_{\gk},  \xi_{\gk})_k$ has uniform bounded variations and converges pointwise to $(x,u,\xi)$,  and that $(x,u,\xi)$ satisfies \eqref{admissible-P} for all $t\in [0,1]$,  
we obtain that the sequence   $\dot{x}_{\gk}$ is of bounded variations and converges pointwise to $\dot{x}\in BV([0,1];\R^n)$. 
Since $(\|\dot{x}_{\gk}\|_{\infty})_k$  is  bounded, we conclude that the sequence $\dot{x}_{\gk}$  also converges {\it strongly} in $L^2$ to $\dot{x}$. Therefore, $x_{\gk}$ converges strongly in the norm topology of $W^{1,2}$ to $x$. \end{proof}

We proceed to present the proof of our approximation result, namely, Theorem \ref{lasthopefullybis}.

\begin{proof}\bp{\it of} Theorem \ref{lasthopefullybis}. We consider $k$ large enough so that $C_0(k)\subset \tilde{C}_0(\delta)$ and $C_1(k)\subset \tilde{C}_1(\delta)$, see (\ref{c01kin}). By Corollary \ref{newstrongconv},  $\x_{\gk}\f\x$ {\it strongly}  in $W^{1,2}$, and hence, for $k$ sufficiently large,  $\x_{\gk}(t)\in \bar{B}_{\delta}(\x(t))$ for all $t\in [0,1]$, and $\y_{\gk}(1)\in[-\delta,\delta]$, where $$\y_{\gk}(t):=\int_0^t \|\dot{\x}_{\gk}(s)-\dot{\x}(s)\|^2\sp ds.$$ Thus, the triplet state $(\x_{\gk},\bar{y}_{\gk},\bar{z}_{\gk}:=0)$ solves $(D_{\gk})$ for $((\bar{c}_k,0,0),\u)$, with   $\x_{\gk}(t)\in \bar{B}_{\delta}(\x(t))$ and  $\u(t)\in U(t)\cap\bar{B}_\delta(\u(t))$, for all $t\in [0,1]$,  and $(\x_{\gk}(1),\y_{\gk}(1), \bar{z}_{\gk}(1)=0) \in C_1(k)\times [-\delta,\delta]\times [-\delta,\delta]$. Therefore, for $k$  sufficiently large, $(\x_{\gk},\y_{\gk},0,\u)$ is an {\it admissible} quadruplet for $(P_{\gk})$.  Using the continuity of $g$ on $\tilde{C}_0(\delta)\times \tilde{C}_1(\delta)$ and the definition of $J(x,u,z,u)$, we obtain that $J(x,u,z,u)$ is bounded from below. Hence, for $k$ large enough, $\inf (P_{\gk})$  is {\it finite}. 

Fix $k$ sufficiently large so that $\inf (P_{\gk})$  is {\it finite}. Let $(x_{\gk}^n,y^n_{\gk},z^n_{\gk},u^n_{\gk})_n\in W^{1,2}([0,1];\R^n)\times AC([0,1];\R)\times AC([0,1];\R)\times\mathscr{W} $  be a minimizing sequence for $(P_{\gk})$, that is, the sequence is admissible for ($P_{\gk})$ and 
\begin{equation} \label{chadin1w} \lim_{n\f\infty} J(x_{\gk}^n,y^n_{\gk},z^n_{\gk},u^n_{\gk})= \inf (P_{\gk}).
\end{equation}
Since for each $n$, $x_{\gk}^n$ solves $(D_{\gk})$ for $(x_{\gk}^n(0),u^n_{\gk})$, and $(x_{\gk}^n(0))_n \in C_0(k)\subset C$, then, by  (\ref{eq2}), we have that the sequence $(x_{\gk}^n)_n$ is uniformly bounded in $C([0,1];\R^n)$ and the sequence $(\dot{x}_{\gk}^n)_n$ is uniformly bounded in $L^2$. On the other hand,  from (A4.2), we have that  sets $U(t)$ are compact and uniformly bounded, then, the sequence $(u^n_{\gk})_n$, which is in $\mathscr{W}$, is uniformly bounded in $C([0,1];\R^m)$. Moreover, its derivative  sequence, $(\dot{u}^n_{\gk})_n$, must be uniformly bounded in $L^2$. Indeed, if this is not true, then  there exists a subsequence  of $\dot{u}^n_{\gk}$, we do not relabel, such that $\lim\limits_{n\f \infty} \|\dot{u}^n_{\gk}\|_2=\infty.$ Using that $g$ is bounded on $\tilde{C}_0(\delta)\times \tilde{C}_1(\delta)$, it follows that \begin{eqnarray*} J(x_{\gk}^n,y^n_{\gk},z^n_{\gk},u^n_{\gk}) &\geq&  \min_{(c_1,c_2)\in \tilde{C}_0(\delta)\times \tilde{C}_1(\delta)}g(c_1,c_2)+\frac{1}{2}z^n_{\gk}(1)\\&=& \min_{(c_1,c_2)\in \tilde{C}_0(\delta)\times \tilde{C}_1(\delta)}g(c_1,c_2)+\frac{1}{2}\|\dot{u}^n_{\gk}-\dot{\u}\|^2_2\end{eqnarray*} and hence, $\lim\limits_{n\f \infty} J(x_{\gk}^n,y^n_{\gk},z^n_{\gk},u^n_{\gk})=\infty,$  contradicting (\ref{chadin1w}). Thus, also $(\dot{u}_{\gk}^n)_n$ is uniformly bounded in $L^2$. Therefore, by Arzel\`a-Ascoli theorem, along a subsequence (we do not relabel), the sequence $(x_{\gk}^n, u_{\gk}^n)_n$ converges  uniformly to a pair $(x_{\gk}, u_{\gk})$  and the sequence $(\dot{x}_{\gk}^n, \dot{u}_{\gk}^n)_n$ converges  weakly in $L^2$ to the pair $(\dot{x}_{\gk},\dot{u}_{\gk})$. \sloppy Hence, $(x_{\gk}, u_{\gk})\in  W^{1,2}([0,1];\R^n)\times \mathscr{W}$. Moreover, the following two inequalities hold \begin{equation}\label{uw12xw12}\|\dot{x}_{\gk}-\dot{\x}\|_2^2\leq \liminf_{n\f \infty} \|\dot{x}^n_{\gk}-\dot{\x}\|_2^2\;\;\hbox{and}\;\; \|\dot{u}_{\gk}-\dot{\u}\|_2^2\leq \liminf_{n\f \infty} \|\dot{u}^n_{\gk}-\dot{\u}\|_2^2. \end{equation}
Since $C_0(k)$, $C_1(k)$,  $\bar{B}_{\delta}(\x(t))$ and $U(t)\cap \bar{B}_{\delta}(\u(t))$ are closed for all $t\in[0,1]$, and from the uniform convergence, as $n\f\infty$,  of the sequence $(x_{\gk}^n, u_{\gk}^n)$ to $(x_{\gk}, u_{\gk})$, we get that the inclusions $x_{\gk}(0)\in C_0(k)$ and  $x_{\gk}(1)\in C_1(k)$, and  $x_{\gk}(t)\in \bar{B}_{\delta}(\x(t))$,  and $u_{\gk}(t)\in U(t)\cap \bar{B}_{\delta}(\u(t))$, for all $t\in[0,1]$. To prove that  $x_{\gk}$ is the solution of  $(D_{\gk})$ corresponding to  $(x_{\gk}(0),u_{\gk})$,  we first  use that  $x_{\gk}^n$ is the solution of $(D_{\gk})$ for $(x_{\gk}^n(0),u^n_{\gk})$,  that is,  for $t\in[0,1]$,
$$ {x}_{\gk}^n(t) =  x_{\gk}^n(0)+\int_0^t \left[{f}_\Vphi(x_{\gk}^n(s),u^n_{\gk}(s))- \gk e^{\gk\psi(x_{\gk}^n(s))} \nabla\psi(x_{\gk}^n(s))\right]\,ds. $$
Using  that  $(x_{\gk}^n(t), u_{\gk}^n(t))\in [C\cap \bar{B}_{\delta}(\x(t))]\times [U(t)\cap \bar{B}_{\delta}(\u(t))]$,   $f_{\Vphi}$ is Lipschitz on $[C\cap \bar{\mathbb{B}}_{\delta}(\x)]\times [(\mathbb{U}+\tilde{\rho}\bar{B}) \cap \bar{\mathbb{B}}_{\delta}(\u)]$, and $(x_{\gk}^n, u^n_{\gk})$  converges uniformly to  $(x_{\gk}, u_{\gk})$, then, upon  taking the limit, as $n\f\infty$, in the last equation we conclude  that $(x_{\gk},u_{\gk})$  satisfies the same equation, that is,
 $$\dot{x}_{\gk}(t)= {f}_\Vphi(x_{\gk}(t),u_{\gk}(t))- \gk e^{\gk\psi(x_{\gk}(t))} \nabla\psi(x_{\gk}(t)),\;\;t\in[0,1]\;\textnormal{a.e.}$$
We define for all $t\in [0,1]$, $$y_{\gk}(t):= \int_0^t \|\dot{x}_{\gk}(\tau)-\dot{\x}(\tau)\|^2\sp d\tau\;\;\hbox{and}\;\;z_{\gk}(t):= \int_0^t \|\dot{u}_{\gk}(\tau)-\dot{\u}(\tau)\|^2\sp d\tau.$$
Clearly we have: \begin{itemize}[leftmargin=*]
 \item $y_{\gk}\in AC([0,1];\R),\;\;\dot{y}_{\gk}(t)=\|\dot{x}_{\gk}(t)-\dot{\x}(t)\|^2,\;\; t\in[0,1] \hbox{ a.e.,}$ and $y_{\gk}(0)=0.$
 \item $z_{\gk}\in AC([0,1];\R),\;\;\dot{z}_{\gk}(t)=\|\dot{u}_{\gk}(t)-\dot{\u}(t)\|^2,\;\; t\in[0,1] \hbox{ a.e.,} $ and $z_{\gk}(0)=0.$
\end{itemize}
Moreover, since $\|\dot{x}^n_{\gk}-\dot{\x}\|_2^2=y^n_{\gk}(1)\in [-\delta,\delta]$ and $\|\dot{u}^n_{\gk}-\dot{\u}\|_2^2= z^n_{\gk}(1)\in [-\delta,\delta]$, the two inequalities of \eqref{uw12xw12} yield that \begin{equation} \label{ykandzk1} y_{\gk}(1)\in [-\delta,\delta]\;\;\hbox{and}\;\;z_{\gk}(1)\in [-\delta,\delta].\end{equation}
Hence, $(x_{\gk},y_{\gk},z_{\gk},u_{\gk})$ is admissible for $({P}_{\gk})$. Now  using \eqref{chadin1w} and the second inequality of \eqref{uw12xw12}, it follows that
\begin{eqnarray*}&&\inf (P_{\gk}) = \lim_{n\f\infty} {J}(x_{\gk}^n,y^n_{\gk},z^n_{\gk},u^n_{\gk})\\ &\bp\bp=&\bp \lim_{n\f\infty}\left(g(x^n_{\gk}(0),x^n_{\gk}(1))+ \frac{1}{2}\left(\|u^{n}_{\gk}(0)-\u(0)\|^2 +z_{\gk}^n(1)+ \|x^n_{\gk}(0)-\x(0)\|^2\right)\bp\right)\\&\bp\bp=&\bp g(x_{\gk}(0),x_{\gk}(1))\bbp+\bbp\frac{1}{2}\|u_{\gk}(0)-\u(0)\|^2 \bbp+\bbp\frac{1}{2}\liminf_{n\f\infty}\|\dot{u}_{\gk}^n-\dot{\u}\|_2^2 +\frac{1}{2}\|x_{\gk}(0)-\x(0)\|^2\\ &\bp\bp\geq&\bp g(x_{\gk}(0),x_{\gk}(1))+\frac{1}{2}\|u_{\gk}(0)-\u(0)\|^2 +\frac{1}{2}\|\dot{u}_{\gk}-\dot{\u}\|_2^2 +\frac{1}{2}\|x_{\gk}(0)-\x(0)\|^2\\ &\bp\bp=&\bp  {J}(x_{\gk},y_{\gk},z_{\gk},u_{\gk}).
 \end{eqnarray*}
Therefore, for each  $k$, large enough,  $(x_{\gk},y_{\gk},z_{\gk},u_{\gk})$ is optimal for $({P}_{\gk})$.

As Remark \ref{c0subc}  asserts that, for $k$ large, $C_0(k)\subset C(k)\subset C$, then,  Lemma \ref{invariance} and Theorem \ref{prop2}$(i)$ yield that the sequence $(x_{\gk}, \xi_{\gk})_k$, where $\xi_{\gk}$ is given via \eqref{defxi},  admits  a subsequence, not relabled, having $(x_{\gk})_k$ converging uniformly to some $x\in W^{1,2}([0,1];\R^n)$  with images in $C$, $(\dot{x}_{\gk}, \xi_{\gk})_k$ converging weakly in $L^2$ to $(\dot{x}, \xi)$ and $\xi$ supported on $I^0(x)$. 

Now, consider  the sequence $(u_{\gk})_k$, which is in $\mathscr{W}$. It has a uniformly bounded derivative in $L^2.$  In fact,   the admissibility of $(\x_{\gk},\bar{y}_{\gk},0,\u)$, and the optimality of $(x_{\gk},y_{\gk},z_{\gk},u_{\gk})$ for $({P}_{\gk})$, imply that  \begin{equation}\label{newimport2w12} {J}
(x_{\gk},y_{\gk},z_{\gk},u_{\gk})\leq g(\x_{\gk}(0),\x_{\gk}(1))+\frac{1}{2}\|\x_{\gk}(0)-\x(0)\|^2.
\end{equation}
This, together with the continuity of $g$ on $\tilde{C}_0(\delta)\times \tilde{C}_1(\delta)$,  the  uniform boundedness of the sequences $(x_{\gk})_k$ and  $(\x_{\gk})_k$, and  the boundedness of $U(0)$,  imply that for some $\hat{M}>0$ we have 
that \begin{eqnarray*} \|\dot{u}_{\gk}-\dot{\u}\|_2^2&\leq& 2 \left(g(\x_{\gk}(0),\x_{\gk}(1))-g(x_{\gk}(0),x_{\gk}(1)\right)+\|\x_{\gk}(0)-\x(0)\|^2\nonumber\\&-&\|u_{\gk}(0)-\u(0)\|^2-\|x_{\gk}(0)-\x(0)\|^2\leq{\hat{M}}\label{ugkbounded}.\end{eqnarray*}
Therefore, $(u_{\gk})_k$ has uniformly bounded derivative in $L^2$. Now since in addition we have that $x_{\gk}(0)\in C_0(k) \subset C(k)$, we are in a position to apply Theorem \ref{prop3new}. We obtain a subsequence (not relabeled) of  $u_{\gk}$, and $u\in \mathscr{W}$ such that $(x_{\gk}, u_{\gk})$ converges uniformly to $(x,u)$, $\dot{u}_{\gk}$ converges weakly in $L^2$ to $\dot{u}$, all the conclusions of Theorem \ref{propooufnew} hold including that $x_{\gk}(t)\in\inte C$ for all $t\in[0,1]$, $(\dot{x}_{\gk}, \xi_{\gk})$ converges {\it strongly} in $L^2$ to  $(\dot{x},\xi)$, $\dot{x}\in BV([0,1];\R^n)$,  $\xi \in BV([0,1];\R^+)$, and,
  for all $t\in [0,1]$, $(x, u,\xi)$ satisfies \eqref{admissible-P}-\eqref{boundxi} and  $x$ uniquely solves $(D)$  for $u$, that is,  
$$\begin{cases} \dot{x}(t)= f_\Vphi(x(t),u(t))-\xi(t) \nabla\psi(x(t))\in{f}(x(t),u(t))- \partial\varphi (x(t)),\;\;\forall\, t\in[0,1],\\x(0)\in C_0\cap \bar{B}_{\delta_{\bbbp o}}(\x(0)). \end{cases} $$
Moreover, we have  \begin{equation}\label{uw12} \|\dot{u}-\dot{\u}\|_2^2\leq \liminf_{k\f \infty} \|\dot{u}_{\gk}-\dot{\u}\|_2^2. \end{equation}
We shall show that $(x,u)$ is admissible for $(P)$. Since $\dot{x}_{\gk}$  converges strongly in $L^2$ to $\dot{x}$, and using \eqref{ykandzk1} and \eqref{uw12}, we have: \begin{itemize}[leftmargin=*]
\item $\displaystyle \|\dot{x}-\dot{\x}\|^2_2=\lim_{k\f\infty}\|\dot{x}_{\gk}-\dot{\x}\|^2_2=\lim_{k\f\infty}y_{\gk}(1)\overset{\eqref{ykandzk1}}{\in}[-\delta,\delta].$
\item $\displaystyle \|\dot{u}-\dot{\u}\|^2_2\overset{\eqref{uw12}}{\leq}\liminf_{k\f\infty}\|\dot{u}_{\gk}-\dot{\u}\|^2_2=\liminf_{k\f\infty}z_{\gk}(1)\overset{\eqref{ykandzk1}}{\in}[-\delta,\delta].$
 \end{itemize}
Hence, $ \|\dot{x}-\dot{\x}\|^2_2\leq \delta$ and $\|\dot{u}-\dot{\u}\|^2_2\leq \delta.$ Since $x_{\gk}(1)\in C_1(k)$, \eqref{limCi}(b) implies that $x(1)\in C_1\cap \bar{B}_{\delta_{\bbbp o}}(\x(0))$. Furthermore, the two inclusions $x_{\gk}(t)\in \bar{B}_{\delta}(\x(t))$ and  $u_{\gk}(t)\in U(t)\cap \bar{B}_{\delta}(\u(t))$, for all $t\in[0,1]$, together with the uniform convergence of $(x_{\gk},u_{\gk})$ to $(x,u)$, give that $x(t)\in \bar{B}_{\delta}(\x(t))$ and $u(t)\in U(t)\cap \bar{B}_{\delta}(\u(t))$, for all $t\in[0,1]$. Therefore, $(x,u)$ is admissible for $({P})$. Hence, the optimality of $(\x,\u)$ for $({P})$  yields that
\begin{equation}\label{mp2w12} g(\x(0),\x(1))\leq g(x(0),x(1)).	
 \end{equation} 
Now,  the uniform convergence of $\x_{\gk}$ to $\x$, (\ref{newimport2w12}), (\ref{mp2w12}), the continuity of $g$,  and  the convergence of $x_{\gk}(0)$ to $x(0)$, imply that \begin{equation} \label{sobolev} u(0)=\u(0)  \;\; \text{and}\;\; \liminf_{k\f \infty}\left(\ \|\dot{u}_{\gk}-\dot{\u}\|_2^2\right)=0,\;\;\hbox{and}\end{equation}  \begin{equation}\label{sobolevbis} x(0)=\x(0)\;\;\hbox{and}\;\;g(\x(0),\x(1)) =g(x(0),x(1)).\end{equation}
The equality (\ref{sobolev}) gives the existence of a subsequence of $u_{\gk}$, we do not relabel, such that  $\dot{u}_{\gk}$ converges {\it strongly }in $L^2$ to $\dot{\u}$. It results that $u_{\gk}$ converges uniformly to $\u$, and hence, $u=\u$. Consequently, $u_{\gk}$ converges strongly in $\mathscr{W}$ to $\u$. Moreover, as $u=\u$, the functions   $x$ and $\x$ solve the dynamic $({D})$ with the same control $\bar{u}$ and  initial condition, see (\ref{sobolevbis}), hence,   by the uniqueness of the solution of $({D})$ we have $x=\x$. Using Lemma \ref{characterizationD}, we obtain that also $\xi=\bar{\xi}$. Therefore, $${x}_{\gk}\xrightarrow[{C([0,1]; \R^n)}]{\textnormal{uniformly}}{\x}\;\;\hbox{and}\;\;(\dot{x}_{\gk}, \xi_{\gk})\xrightarrow[{L^2([0,1]; \R^n\times \R^{+})}]{\textnormal{strongly}}(\dot{\x},\bar{\xi}).$$
This yields that $(y_{\gk},z_{\gk})\f(0,0)$ in the strong topology of ${W^{1,1}([0,1];\R^+\times\R^+)}.$

Since $x_{\gk}(1)\in\left[\left(C_1\cap \bar{B}_{\delta_{\bbbp o}}(\x(1))\right)-\x(1)+\x_{\gk}(1)\right]\cap (\inte C)$ and $\x_{\gk}(1)$ converges to $ \x(1)$,  it follows that  $x_{\gk}(1)\in \big[\left(C_1\cap \bar{B}_{\delta_{\bbbp o}}(\x(1))\right)+\tilde{\rho}{B}\big]\cap(\inte C)$, for $k$ sufficiently large. On the other hand, the definition of $C_0(k)$ and the convergence of  $\rho_{k}$ to  $0$ yield that, for $k$  large enough,  $x_{\gk}(0)\in \big[\left(C_0\cap \bar{B}_{\delta_{\bbbp o}}(\x(0))\right)+\tilde{\rho}{B}\big]\cap(\inte C)$. This terminates the proof of Theorem \ref{lasthopefullybis}. \end{proof}

In the next lemma, a compactness result is derived  for  the solutions  of ($D$), where the controls  are restricted to be in $\mathscr{W}$ and  $x(1)\in C_1$.   The equivalence between $(D)$ and   equations \eqref{admissible-P}-\eqref{boundxi}  is employed.    

\begin{lemma}[Compact  trajectories and  controls  for $(D)$] \label{compact}  Assume  that  \textnormal{(A1)-(A4.3)} hold. Let  
 $(x_j,u_j)_j$ be a sequence in $W^{1,\infty}\times\mathscr{W}$ satisfying $(D)$ with $x_j(1)\in C_1$, for all $j\in \N$,   and   $(\|\dot{u}_j\|_2)_j$  be bounded. Consider $(\xi_j)_j$ the corresponding sequence  in $L^{\infty}([0,1];\R^+)$ obtained via \textnormal{Lemma \ref{characterizationD}}, that is,  
 $(x_j,u_j, \xi_j)$  satisfies \eqref{admissible-P}-\eqref{boundxi}, for all $j$. Then there exist a subsequence  of $(x_j,u_j, \xi_j)_j$, we do not relabel, and $(x,u, \xi)\in W^{1,\infty}([0,1];\R^n)\times \mathscr{W} \times L^{\infty}([0,1];\R^+)$  such that 
$(x_j,u_j)_j$ converges uniformly to $(x,u)$,  $(\dot{x}_{j}, \xi_{j})_j$  now converges {\it pointwise} to $(\dot{x}, {\xi})\in BV([0,1];\R^n)\times BV([0,1];\R^+)$, $\dot{u}_j$ converge weakly in $L^2$ to $\dot{u},\xi,$ and $(x,u,\xi)$ satisfies \eqref{admissible-P}-\eqref{boundxi} with $x(1)\in C_1$. In particular, $(x,u)$ is admissible for $(P)$ and $(x_j)_j$ converges to $x$ strongly in the norm topology of $W^{1,2}([0,1];\R^n)$.  \end{lemma}
 \begin{proof}   Using  \eqref{lipx} in Remark \ref{lip.D} for the sequence $(x_j)_j$,   the boundedness of $(\|\dot{u}_j \|_2)_j$,  that $u_j(t))\in U(t)$ for all $t\in [0,1]$,  and  that the sets $U(t)$ are compact and uniformly bounded, by (A4.2),  then  Arzela-Ascoli’s theorem produces  a subsequence, we do not relabel, of  $(x_j , u_j )_j$, that converges uniformly to an absolutely continuous pair $(x, u)$ with $(x(t),u(t)\in C\times U(t)$ for all $t\in [0,1]$, and $(\dot{x}_{j}, \dot{u}_{j})_j$ converging weakly in $L^2$ to $(\dot{x}, \dot{u})$. As for all $j\in\N$,  $x_j(0)\in C_0$ and $x_j(1)\in C_1$, then (A4.1) and (A4.3) yield that  $x(0)\in C_0$ and $x(1)\in C_1$. Using Corollary \ref{newstrongconv}, we obtain $\|\xi_{j}\|_{\infty}\leq \frac{\bar{M}}{2\eta}$, $\xi_{j}\in BV([0,1];\R^+),$ and $V_0^1(\xi_{j})\leq \tilde{M}_2,$  where $\tilde{M}_2$ depends on the uniform bound of $(\|\dot{u}_j\|_2)_j$. By Helly's first theorem, $(\xi_{j})_j$ convergence {\it pointwise} to ${\xi}\in BV([0,1];\R^+).$  On the other hand,  Corollary \ref{newstrongconv} also gives that  \eqref{admissible-P} holds for all $t\in[0,1]$, that is, 
   \begin{equation}\label{forallt}
\dot{x}_{j}(t)=f_{\Vphi}(x_{j}(t),u_{j}(t))-\xi_{j}(t)\nabla\psi(x_{j}(t)),\;\;\forall\sp t\in [0,1].\end{equation} Thus, upon taking the pointwise limit as $j\f \infty$ in \eqref{forallt},
it follows that $(\dot{x}_j)_j$ converges pointwise to its $L^2$-limit $\dot{x}$, and hence, $\dot{x}\in BV([0,1];\R^n)$.  As $(x_j,u_j)$ solves ($D$),  \eqref{lipx} yields 
that $(\|\dot{x}_j\|_\infty)_j$ is uniformly bounded,  and hence,  $(\dot{x}_j)_j$  converges to $\dot{x}$ {\it strongly} in $L^2$.
 
 We  now show that $\xi(t)$ is supported in $ I^{0}(x)$. Let $t\in I^{\-}(x)$ be fixed,  that is, $x(t)\in \inte C$.  Since $(x_j)_j$ converges uniformly to $x$, then we can find $\delta_o >0$ and $j_o \in \N$ such that, for all $ s\in (t-\delta,t+\delta)\cap [0,1]$ and for all $j\ge j_o$, we have $ x_j(s)\in \inte C$, and hence, as  $\xi_j$ satisfies \eqref{boundxi},   $\xi_j(s)=0$. Thus, $\xi_j(s)\f 0$   for $ s\in (t-\delta_o,t+\delta_o)\cap [0,1]$, and whence, $\xi(t)=0$, proving that $\xi $ is supported in $I^{0}(x)$. 
 Therefore,   applying  Lemma \ref{characterizationD}  to $(x,u,\xi)$, we  conclude that $(x, u)$ solves ($D$), $\xi\in L^{\infty}([0,1];\R^{+})$ and  $(x,u,\xi)$ satisfies \eqref{boundxi}. \end{proof}

In the following remark, we provide important information about the constraint qualification property (CQ).

\begin{remark}\label{CQremark} \qquad \begin{enumerate}[$(i)$,leftmargin=0.8cm] 
\item Let $F\colon [0,1]\rightrightarrows\R^m$ be a lower semicontinuous multifunction with closed and nonempty values. For $d_F(t,x):=d(x,F(t))$, we have from \cite[Proposition 2.3]{lowen91} that for  $t\in [0,1]$ and $x\in F(t)$, $\conv (\bar{N}^L_{F(t)}(x))$ is pointed if and only if $0\not\in \partial_x^{\sp >}d_F(t,x)$. The notion of $\partial_x^{\sp >}g(t,x)$  is introduced  for a general function $g(t,x)$ by Clarke in \cite[p.121]{clarkeold}.  For $g(t,x):= d_F(t,x)$,  it is shown in \cite[Corollary 2.2]{lowen91} that   \begin{eqnarray*} &&\partial_x^{\sp >}d_F(t,x)=\\[3pt]&& \conv \left\{\zeta : \zeta=\lim_{i\f\infty}\zeta_i,\;\|\zeta_i\|=1,\;\zeta_i\in N_{F(t_i)}^P(x_i)\;\hbox{and}\;(t_i,x_i)\xrightarrow{{\Gr F}} (t,x)\right\},\end{eqnarray*} where $(t_i,x_i)\xrightarrow{{\Gr F}} (t,x)$ signifies that $(t_i,x_i)\f (t,x)$ with $x_i\in F(t_i)$ for all $i$. Therefore, for  $h\in C([0,1];F)$, we have that $F$ satisfies the constraint qualification (CQ) at $h$ if and only if $0\not\in \partial_x^{\sp >}d_F(t,h(t))$ for all $t\in [0,1]$. Note that the multifunction $(t,x)\mapsto \partial_x^{\sp >}d_F(t,x)$ is uniformly bounded with compact and convex values, and has a closed graph.
  
\item Using the proximal normal inequality, one can easily extend the   arguments in the proof of \cite[Proposition 2.3(d)]{lowen91}, to show  that if the lower semicontinuous multifunction $F$ has closed and ${r}$-{\it prox-regular} values, for some ${r}>0$, (as opposed to convex), then $\conv (\bar{N}^L_{F(t)}(\cdot))={N}^P_{F(t)}(\cdot)={N}^L_{F(t)}(\cdot)={N}_{F(t)}(\cdot),$ and this cone is pointed at $x\in F(t)$ if and only if $F(t)$ is epi-lipschitz at $x$, see \cite[Theorem 7.3.1]{clarkeold} and \cite[Exercise 9.42]{rockwet}. Hence, a lower semicontinuous multifunction $F\colon [0,1]\rightrightarrows\R^m$ with values that are closed and ${r}$-prox-regular,  satisfies the constraint qualification (CQ) at $h\in C([0,1];F)$ if and only if $F(t)$ is epi-Lipschitz at $h(t)$, for all $t\in[0,1]$.
\item If $F(t)=F$ for all $t\in [0,1]$, where $F$ is closed, then $\conv (\bar{N}^L_{F(t)}(\cdot))={N}_{F}(\cdot)$, and this cone is pointed at $x\in F$ if and only if $F$ is epi-Lipschitz at $x$. Hence, a constant multifunction $F$ satisfies the constraint qualification (CQ) at $h\in C([0,1];F)$ if and only if $F$ is epi-Lipschitz at $h(t)$ for all $t\in [0,1]$.
\end{enumerate}
\end{remark}

We terminate this section by the following technical lemma used in the proof of the \enquote{In addition} part of the weak maximization condition of Theorem \ref{mpw12bisw12}. The proof of the lemma follows from the local property of the normal cones, the proximal normal inequality, and the fact that the proximal, Mordukhovich, and Clarke normal cones coincide in our setting.

\begin{lemma} \label{newlemmaprox} Let $F\colon [0,1]\rightrightarrows\R^m$ be a lower semicontinuous multifunction with closed and nonempty values and let $h\in C([0,1];F)$. If there exist $\e_{\bbbp o}>0$ and $r>0$ such that $F(t)\cap \bar{B}_{\e_{\bbbp o}}(h(t))$ is $r$-prox-regular for all $t\in[0,1]$, then for any $\delta>0$ we have \begin{equation*}\label{neweqprox} \conv (\bar{N}^L_{F(t)\cap \bar{B}_{\delta}(h(t))}(h(t)))=N^P_{F(t)\cap \bar{B}_{\e_{\bbbp o}}(h(t))}(h(t)),\;\;\forall t\in [0,1]. \end{equation*}
Moreover, for all $t\in [0,1]$ and for all $\zeta\in \conv (\bar{N}^L_{F(t)\cap \bar{B}_{\delta}(h(t))}(h(t)))$, we have \begin{equation*}\label{neweqproxbis}\textstyle   \<\zeta,v-h(t)\>\leq \frac{\|\zeta\|}{\min\{\e_{\bbbp o},2r\}}\|v-h(t)\|^2,\;\;\forall v\in F(t).	
 \end{equation*}
\end{lemma}

\end{document}